\documentclass[11pt]{amsart}
\usepackage{times}
\usepackage[pdftex]{graphicx}
\usepackage{amssymb,amsfonts,amsmath,amsthm}
\usepackage{float}

\usepackage{color}

\newtheorem{theorem}{Theorem}[section]
\newtheorem{proposition}[theorem]{Proposition}
\newtheorem{lemma}[theorem]{Lemma}
\newtheorem{definition}[theorem]{Definition}
\newtheorem{remark}[theorem]{Remark}
\newtheorem{corollary}[theorem]{Corollary}

\newcommand\indices[2]{\footnotesize{\left\vert \begin{array}{l} 
\hspace{-6pt}  #1 \atop \hspace{-6pt} #2 \end{array} \right.}}

\newcommand\cP{\mathcal P}

\def\balpha{\bar \alpha}
\def\bX{\bar X}

\def\tX{\tilde X}
\def\Ib{{\mathcal I}^b}
\def\Is{{\mathcal I}^{\sigma}}
\def\If{{\mathcal I}^f}
\def\Ig{{\mathcal I}^g}
\def\Ibp{{\mathcal I}^{b,\prime}}
\def\Isp{{\mathcal I}^{\sigma,\prime}}
\def\Ifp{{\mathcal I}^{f,\prime}}
\def\Igp{{\mathcal I}^{g,\prime}}

\newcommand\bA {\mathbb A}

\newcommand\EE {\mathbb E}

\newcommand\HH {\mathbb H}

\newcommand\RR {\mathbb R}

\newcommand\PP {\mathbb P}

\newcommand\bF {\mathbb F}
\newcommand\cA {\mathcal A}

\newcommand\cE {\mathcal E}
\newcommand\cF {\mathcal F}

\def\h{\hat}
\def\t{\tilde}
\def\o{\bar}
\def\u{\underline}

\newcommand\PXt {{\mathbb P}_{X_t}}
\newcommand\PXT {{\mathbb P}_{X_T}}

\begin{document}

\title[FBSDEs and McKean Vlasov]{Forward-Backward Stochastic Differential Equations and Controlled McKean Vlasov Dynamics}

\author{Ren\'e Carmona}
\address{ORFE, Bendheim Center for Finance, Princeton University,
Princeton, NJ  08544, USA.}
\email{rcarmona@princeton.edu}
\thanks{Partially supported  by NSF: DMS-0806591}

\author{Fran\c{c}ois Delarue}
\address{Laboratoire Jean-Alexandre Dieudonn\'e, 
Universit\'e de Nice Sophia-Antipolis, 
Parc Valrose, 
06108 Cedex 02, Nice, FRANCE}
\email{Francois.Delarue@unice.fr}

\subjclass[2000]{Primary }


\date{June 10, 2011}

\begin{abstract}
The purpose of this paper is to provide a detailed probabilistic analysis of the optimal control of nonlinear stochastic dynamical systems
of the McKean Vlasov type. Motivated by the recent interest in mean field games, we highlight the connection and the differences 
between the two sets of problems. We prove a new version of the stochastic maximum principle and give sufficient conditions for existence of an optimal
control. We also provide examples for which our sufficient conditions for existence of an optimal solution are satisfied. Finally we show that our solution to the control problem provides approximate equilibria for large stochastic games with mean field interactions.
\end{abstract}

\maketitle

\section{\textbf{Introduction}}
\label{se:intro}

The purpose of this paper is to provide a detailed probabilistic analysis of the optimal control of nonlinear stochastic dynamical systems
of the McKean-Vlasov type. The present study is motivated in part by a recent surge of interest in mean field games.

We prove a  version of the stochastic Pontryagin maximum principle that is tailor-made to McKean-Vlasov dynamics and give sufficient conditions for existence of an optimal
control. We also provide a class of examples for which our sufficient conditions for existence of an optimal solution are satisfied. Putting these conditions to work at the solution of an optimal control problem leads to the solution of a system of Forward Backward Stochastic Differential Equations (FBSDEs for short) where the marginal distributions of the solutions appear in the coefficients of the equations. We call these equations mean field FBSDEs, or FBSDEs of McKean-Vlasov type. To the best of our knowledge, these equations have not been studied before. A rather general existence result was recently proposed in \cite{CarmonaDelarue_ecp}, but one of the assumptions (boundedness of the coefficients with respect to the state variable) precludes applications to solvable models such that the Linear Quadratic (LQ for short) models. Here, we take advantage of the convexity of the underlying Hamiltonian 
and apply the so-called \emph{continuation method} exposed in \cite{PengWu}
in order to prove existence and uniqueness of the solution of the FBSDEs at hand, extending and refining the results of \cite{CarmonaDelarue_ecp} to the models considered in this paper.   The technical details are given in Section \ref{se:fbsde}.

Without the control $\alpha_t$ (see Eq. \eqref{fo:mkvsde} right below), stochastic differential equations of McKean-Vlasov type are associated to special nonlinear Partial Differential Equations (PDEs) put on a rigorous mathematical footing by Henry McKean Jr in \cite{McKean1}. See also \cite{McKean2,Sznitman,JourdainMeleardWoyczynski}. Existence and uniqueness results for these equations have been developed in order to provide effective equations for studying large systems, reducing the dimension and the complexity, at the cost of handling non Markovian dynamics depending upon the statistical distribution of the solution. In the same spirit, we show in Section   
\ref{se:chaos} that our solution of the optimal control of McKean-Vlasov stochastic differential equations provides strategies putting a large system of individual optimizers in an approximate equilibrium, the notion of equilibrium being defined appropriately. The proof is based on standard arguments from the theory of the propagation of chaos, see for example \cite{Sznitman,JourdainMeleardWoyczynski}. The identification of approximate equilibriums in \textit{feedback form}
 requires strong regularity properties of the decoupling field of the FBSDE. They are proved in Section \ref{se:fbsde}. 

\section{\textbf{Probabilistic Set-Up of McKean-Vlasov Equations}}
\label{se:mkv}

In what follows, we assume that $W=(W_t)_{0\le t\le T}$ is an $m$-dimensional standard Wiener process defined on a probability space $(\Omega,\cF,\PP)$ and $\bF=(\cF_t)_{0\le t\le T}$ is its natural filtration possibly augmented 
with an independent $\sigma$-algebra ${\mathcal F}_{0}$.  For each random variable/vector or stochastic process $X$, we denote
by $\PP_X$ the law (alternatively called the distribution) of $X$. 

The stochastic dynamics of interest in this paper are given by a stochastic process
$ X=(X_t)_{0\le t\le T}$ satisfying a nonlinear stochastic differential equation of the form
\begin{equation}
\label{fo:mkvsde}
dX_t=b(t,X_t,\PXt, \alpha_t)dt+\sigma(t,X_t,\PXt,\alpha_t)dW_t, \qquad\qquad 0\le t\le T,
\end{equation}
where the drift and diffusion coefficient of the state $X_t$ of the system are given by the pair of deterministic functions $(b,\sigma):[0,T]\times\RR^d\times \cP_2(\RR^d)\times A\rightarrow \RR^{d} \times \RR^{d \times m}$ and $ \alpha=(\alpha_t)_{0\le t\le T}$ is a progressively measurable process with values in a measurable space $(A,\cA)$.
Typically, $A$ will be an open subset of an Euclidean space $\RR^k$ and $\cA$ the $\sigma$-field induced by the Borel $\sigma$-field of this Euclidean space. 

Also, for each measurable space $(E,\cE)$, we use the notation $\cP(E)$ for the space of probability measures on $(E,\cE)$, assuming that the $\sigma$-field $\cE$ on which the measures are defined is understood. When $E$ is a metric or a normed space (most often $\RR^d$
), we denote by $\cP_p(E)$ the subspace of $\cP(E)$ of the probability measures of order $p$, namely those probability measures which integrate the $p$-th power of the distance to a fixed point (whose choice is irrelevant in the definition of $\cP_p(E)$).
The term \emph{nonlinear} does not refer to the fact that the coefficients $b$ and $\sigma$ could be nonlinear functions of  $x$ but instead to the fact that they depend not only on the value of the unknown process $X_t$ at time $t$, but also on its marginal distribution $\PP_{X_t}$.
We shall assume that the drift coefficient $b$ and the volatility $\sigma$ satisfy the following assumptions.
\begin{itemize}
\item[(A1)] \hskip 6pt For each $x\in\RR^d$, $\mu\in\cP_2(\RR^d)$ and $\alpha\in A$, the function 
$[0,T] \ni t \mapsto (b,\sigma)(t,x,\mu,\alpha) \in \RR^d \times \RR^{d \times m}$ is square integrable;
\item[(A2)] \hskip 6pt  $\exists c>0$, $\forall t\in[0,T]$,  $\forall \alpha\in A$,  $\forall x,x'\in\RR^d$,  $\forall \mu,\mu'\in\cP_2(\RR^d)$,
$$
|b(t,x,\mu,\alpha)-b(t,x',\mu',\alpha)| +  |\sigma(t,x,\mu,\alpha)-\sigma(t,x',\mu',\alpha)|\le c \bigl[ 
|x-x'|+W_2(\mu,\mu') \bigr];
$$
\end{itemize} 
where $W_2(\mu,\mu')$ denotes the 2-Wasserstein distance. For a general $p>1$, the p-Wasserstein distance $W_p(\mu,\mu')$ is defined on $\cP_p(E)$ by:
$$
W_p(\mu,\mu')=\inf\left\{\left[\int_{E \times E} |x-y|^p\pi(dx,dy)\right]^{1/p};\;\pi\in\cP_2(E\times E) \text{ with marginals } \mu \text{ and } \mu'\right\}.
$$
Notice that if $X$ and $X'$ are random variables of order $2$, then  $W_2(\PP_X,\PP_{X'})\le [\EE|X-X'|^2]^{1/2}$.

The set $\bA$ of so-called \textit{admissible} control processes $ \alpha$ is defined as the set of $A$-valued progressively measurable processes $\alpha \in \HH^{2,k}$, where $\HH^{2,n}$ denotes the Hilbert space
$$
\HH^{2,n}:=\Big\{Z\in\HH^{0,n};\;\EE\int_0^T|Z_s|^2ds<+\infty\Big\}
$$
with $\HH^{0,n}$ standing for the collection of all $\RR^n$-valued progressively measurable processes on $[0,T]$. By (A.1) and (A.2), any $\alpha \in \bA$ satisfies 
$$
\EE\int_0^T[|b(t,0,\delta_0,\alpha_t)|^2+|\sigma(t,0,\delta_0,\alpha_t)|^2]dt<+\infty. 
$$
Together with the Lipschitz assumption (A.2), this guarantees that, for any $\alpha\in\bA$, there exists a unique solution $ X= X^{\alpha}$ of 
\eqref{fo:mkvsde}, and that moreover this solution satisfies
\begin{equation}
\label{fo:moment}
\EE\sup_{0\le t\le T}|X_t|^p<+ \infty
\end{equation}
for every $p \in [1,2]$. See e.g. \cite{Sznitman,JourdainMeleardWoyczynski} for a proof.
\vskip 2pt
The stochastic optimization problem which we consider is to minimize the objective function
\begin{equation}
\label{fo:mkvobjective}
J(\alpha)=\EE\left\{\int_0^Tf(t,X_t,\PXt, \alpha_t)dt+g(X_T,\PXT)\right\}  ,
\end{equation}
over the set $\bA$ of admissible control processes $\alpha=(\alpha_t)_{0\le t\le T}$. The \emph{running cost} function $f$ is a real valued deterministic function on $[0,T]\times\RR^d\times \cP_2(\RR^d)\times A$, and the  \emph{terminal cost} function $g$ is a real valued deterministic function on $\RR^d\times \cP_2(\RR^d)$. Assumptions on the cost functions $f$ and $g$ will be spelled out later.

\vskip 4pt
The McKean-Vlasov dynamics posited in \eqref{fo:mkvsde} are sometimes called of \emph{mean field  type}. This is justified by the fact that the uncontrolled stochastic differential equations of the McKean-Vlasov type first appeared in the infinite particle limit of large systems of particles with mean field interactions. See \cite{McKean2, Sznitman,JourdainMeleardWoyczynski} for example. 
Typically, the dynamics of such a system of $N$ particles are given by a system of $N$ stochastic differential equations of the form
$$
dX^i_t=b^i(t,X_t^1,\cdots,X_t^N)dt +\sigma^i(t,X_t^1,\cdots,X_t^N)dW^i_t,
$$ 
where the $W^i$'s are $N$ independent standard Wiener processes in $\RR^{m}$,  the $\sigma^i$'s are $N$ 
deterministic functions from $[0,T]\times\RR^{N\times d}$ into the space of $d\times m$ real matrices, and the $b^i$'s are $N$ 
deterministic functions from $[0,T]\times\RR^{N\times d}$ into $\RR^{d}$. The interaction between the particles is said to be of mean field type when the functions $b^i$ and $\sigma^i$ are of the form
$$
b^i(t, x_1,\cdots,x_N)=b(t, x_i,\bar\mu^N_{\u x}),\quad\text{and}\qquad \sigma^i(t, x_1,\cdots,x_N)=\sigma(t, x_i,\bar\mu^N_{\u x}),\quad i=1,\cdots,N,
$$
for some deterministic function $b$ from $[0,T]\times\RR^{d}\times \cP_1(\RR^d)$ into $\RR^{d}$, and $\sigma$ from $[0,T]\times\RR^{d}\times \cP_1(\RR^d)$ into the space of $d\times m$ real matrices. Here, for each $N$-tuple $\u x=(x_1,\cdots,x_N)$, we denote by $\bar{\mu}^N_{\u x}$, or $\bar{\mu}^N$ when no confusion is possible, the empirical probability measure defined by:
\begin{equation}
\label{fo:empirical}
\bar{\mu}^N(dx')=\frac1N\sum_{j=1}^N\delta_{x_j}(dx')
\end{equation}
and for each $x$ by $\delta_x$ the unit point mass (Dirac) measure at $x$. We shall come back to this formulation of the problem 
in the last section of the paper when we use results from the propagation of chaos to construct approximate equilibriums.
\vspace{4pt}

We emphasize that the optimization problem \eqref{fo:mkvobjective} differs from the optimization problem encountered in the theory of mean field games. Differences between these optimization problems are discussed in \cite{CarmonaDelarueLaChapelle}. When handling mean field games, the optimization of the cost functional \eqref{fo:mkvobjective} is performed for a fixed flow of probability measures. In other words, the argument $(\PP_{X_{t}})_{0 \leq t \leq T}$ in \eqref{fo:mkvsde}
and \eqref{fo:mkvobjective}
is kept fixed as $\alpha$ varies and the controlled processes are driven by the same flow of measures, which is not necessarily the flow of distributions of the process $(X_{t})_{0 \leq t \leq T}$ but only an input. Solving the corresponding mean field game then consists in identifying a flow of probability measures, that is an input, such that the optimal states have precisely the input as flow of statistical distributions. As highlighted in Section \ref{se:chaos}, the optimization problem for controlled McKean-Vlasov dynamics as we consider here, also reads as a limit problem as $N$ tends to infinity, of the optimal states of $N$ interacting players or agents using a common policy.

\subsection*{Useful Notations} Given a function $h : \RR^d \rightarrow \RR$ and a vector $p \in \RR^d$, we will denote by 
$\partial h(x) \cdot p$ the action of the gradient of $h$ onto $p$. When $h : \RR^d \rightarrow \RR^{\ell}$, we will also denote by 
$\partial h(x) \cdot p$ the action of the gradient of $h$ onto $p$, the resulting quantity being an element of $\RR^{\ell}$. When  $h : \RR^d \rightarrow \RR^{\ell}$ and $p \in \RR^{\ell}$, we will denote by $\partial h(x) \odot p$ the element of $\RR^d$ defined by 
$\partial_{x} [ h(x) \cdot p]$ where $\cdot$ is here understood as the inner product in $\RR^{\ell}$.

\section{\textbf{Preliminaries}}
\label{se:preliminaries}
We now introduce the notation and concepts needed for the analysis of the stochastic optimization problem associated with the control of McKean-Vlasov dynamics.

\subsection{Differentiability and Convexity of Functions of Measures}
There are many notions of differentiability for functions defined on spaces of measures, and recent progress in the theory of optimal transportation
have put several forms in the limelight. See for example \cite{AmbroglioGigliSavare,Villani} for expos\'es of these geometric approaches in textbook form.
However, the notion of differentiability which we find convenient for the type of stochastic control problem studied in this paper is slightly different. It is more of a functional analytic nature. We believe that it was introduced by P.L. Lions in his lectures at the \emph{Coll\`ege de France}. See \cite{Cardaliaguet} for a readable account. This notion of differentiability is based on the \emph{lifting} of functions $\cP_2(\RR^d)\ni\mu\mapsto H(\mu)$ 
into functions $\t H$  defined on the Hilbert space $L^2(\t\Omega;\RR^d)$ over some probability space $(\t\Omega,\t\cF,\t\PP)$ by setting $\t H(\tX)=H(\t\PP_{\tX})$ for $\t X \in L^2(\t \Omega;\RR^d)$, $\tilde{\Omega}$ being a Polish space and $\t \PP$ an atomless measure. Then, a function $H$
is said to be differentiable  at $\mu_0\in\cP_2(\RR^d)$ if there exists a random variable $\tX_0$ with law $\mu_0$, in other words 
satisfying $\t\PP_{\tX_0}=\mu_0$, such that the lifted function $\t H$ is Fr\'echet differentiable at $\tX_0$. 
Whenever this is the case, the Fr\'echet derivative of $\t H$ at $\t X_0$ can be viewed as an element of $L^2(\t\Omega;\RR^d)$ by identifying $L^2(\t\Omega;\RR^d)$ and its dual. It then turns out that its distribution depends only upon the law $\mu_0$ and not upon the particular random variable $\t X_0$ having distribution $\mu_0$. See Section 6 in \cite{Cardaliaguet} for details. This Fr\'echet derivative $[D\t H](\t X_0)$ is called the representation of the derivative of $H$ at $\mu_0$ along the variable $\t X_{0}$.  Since it is viewed as an element of $L^2(\t\Omega;\RR^d)$, by definition, 
\begin{equation}
\label{eq:1:1:1}
H(\mu)=H(\mu_0)+[D \t H](\t X_0)\cdot (\t X-\t X_0) +o(\|\t X-\t X_0\|_{2})
\end{equation}
whenever $\tX$ and $\tX_0$ are random variables with distributions $\mu$ and $\mu_0$ respectively, the dot product being here the $L^2$- inner product over $(\t \Omega,\t {\mathcal F},\t \PP)$ and $\| \cdot \|_{2}$ the associated norm. It is shown in \cite{Cardaliaguet} that, as a random variable, it is of the form $\t h(\t X_0)$ for some deterministic measurable function $\t h : \RR^d \rightarrow \RR^d$, which is uniquely defined $\mu_0$-almost everywhere on $\RR^d$. The equivalence class of 
$\t h$ in $L^2(\RR^d,\mu_{0})$ being uniquely defined, we can denote it by $\partial_{\mu} H(\mu_{0})$ (or $\partial
 H(\mu_{0})$ when no confusion is possible): We will call $\partial_\mu H(\mu_0)$ the derivative of $H$ at $\mu_{0}$
 and we will 
 often identify it
  with a function $\partial_{\mu} H(\mu_{0})( \, \cdot \, ) : \RR^d \ni x \mapsto \partial_{\mu} H(\mu_{0})(x)
\in \RR^d$ (or by $\partial H(\mu_{0})( \, \cdot \, )$ when no confusion is possible). Notice that $\partial_{\mu} H(\mu_{0})$ allows us to express $[D \t H](\t X_0)$ as a function of any random variable $\t X_{0}$ with distribution $\mu_0$, irrespective of where this random variable is defined. In particular, 
the differentiation formula
\eqref{eq:1:1:1} is somehow invariant by modification of the space $\tilde{\Omega}$ and of the variables $\t X_{0}$ and $\t X$ 
used for the representation of $H$,
in the sense that $[D \t H](\t X_0)$ always reads as $\partial_{\mu}H(\mu_{0})(\t X_{0})$, whatever the choices of 
$\t \Omega$, $\t X_{0}$ and $\t X$ are.
It is plain to see how this works when the function $H$ is of the form
\begin{equation}
\label{fo:Hofmu}
H(\mu)=\int_{\RR^d} h(x)\mu(dx)=\langle h,\mu\rangle
\end{equation}
for some scalar differentiable function $h$ defined on $\RR^d$. Indeed, in this case, $\t H(\t X)=\t\EE [h(\t X)]$ and $D\t H(\t X)\cdot \t Y=\t\EE[\partial h(\t X)\cdot \t Y]$ so that we can think of $\partial_\mu H(\mu)$ as the deterministic function $\partial h$. We will use this particular example to recover the Pontryagin principle originally derived in \cite{AndersonDjehiche} for scalar interations as a particular case of the general Pontryagin principle which we prove below. The example \eqref{fo:Hofmu} highlights the fact that this notion of differentiability is very different from the usual one. Indeed, given the fact that the function $H$ defined by \eqref{fo:Hofmu} is linear in the measure $\mu$ when viewed as an element of the dual of a function space, one should expect the derivative to be $h$ and NOT $h'$ ! The notion of differentiability used in this paper is best understood as differentiation of functions of limits of empirical measures (or linear combinations of Dirac point masses) in the directions of the atoms of the measures. We illustrate this fact in the next two propositions.

\begin{proposition}
\label{pr:diff_empir}
If $u$ is differentiable on ${\mathcal P}_{2}(\RR^d)$,  given any integer $N \geq 1$, we define the \textit{empirical} projection of $u$ onto $\RR^d$ by
\begin{equation*}
\bar{u}^N : (\RR^d)^N \ni \u x= (x_{1},\dots,x_{N}) \mapsto u \biggl( \frac{1}{N} \sum_{i=1}^N \delta_{x_{i}} \biggr).
\end{equation*}
Then, $\bar{u}^N$ is differentiable on $(\RR^d)^N$ and, for all $i \in \{1,\dots,N\}$,
\begin{equation*}
\partial_{x_{i}} \bar{u}^N(\u x) = 
\partial_{x_{i}} \bar{u}^N(x_{1},\dots,x_{N}) = \frac{1}{N} \partial u \biggl( \frac{1}{N} \sum_{j=1}^N \delta_{x_{j}}\biggr)(x_{i}).
\end{equation*}
\end{proposition}

\begin{proof} On $(\tilde{\Omega},\t {\mathcal F},\t \PP)$, consider a 
uniformly distributed
random variable $\vartheta$  over the set $\{1,\dots,N\}$. Then, for any fixed
$\u x = (x_{1},\dots,x_{N}) \in (\RR^d)^N$, $x_{\vartheta}$
is a random variable  having the distribution $\bar{\mu}^N = N^{-1} \sum_{i=1}^N \delta_{x_{i}}$. In particular, with the same notation as above for $\t u$, 
\begin{equation*}
\bar{u}^N(\u x)
=
\bar{u}^N(x_{1},\dots,x_{N}) = \tilde u (x_{\vartheta}). 
\end{equation*}
Therefore, for $\u h = (h_{1},\dots,h_{N}) \in (\RR^d)^N$, 
\begin{equation*}
\bar{u}^N( \u x + \u h) = \tilde{u}(x_{\vartheta}+h_{\vartheta})
 = \tilde{u}(x_{\vartheta}) + D\tilde{u}(x_{\vartheta}) \cdot h_{\vartheta} + o(\vert h \vert),
\end{equation*}
the dot product being here the $L^2$- inner product over $(\t \Omega,\t {\mathcal F},\t \PP)$, from which we deduce
\begin{equation*}
\bar{u}^N(\u x + \u h)
= \bar{u}^N(\u x) + \frac{1}{N} \sum_{i=1}^N \partial\, u(\o\mu^N)(x_i) h_{i} + o(\vert h \vert),
\end{equation*}
which is the desired result.
\end{proof}

\vspace{5pt}

The mapping $D\tilde{u} : L^2(\t \Omega;\RR^d) \rightarrow L^2(\t \Omega;\RR^d)$ is said to be Lipchitz continuous if there exists a constant $C>0$ such that, for any square integrable random variables $\t X$ and $\t Y$ in $L^2(\t \Omega;\RR^d)$, it holds $\|D\tilde{u}(\t X)-D\tilde{u}(\t Y)\|_{2}\le C\|\t X-\t Y\|_{2}$. In such a case, the Lipschitz property can be transferred onto $L^2(\Omega)$ and then rewritten as:
\begin{equation}
\label{fo:Lip}
{\mathbb E} \bigl[ \vert \partial u(\PP_{X})(X) - \partial u(\PP_{Y})(Y) \vert^2 \bigr] \leq 
C^2 \EE\bigl[ \vert X-Y \vert^2 \bigr],
\end{equation}
for any square integrable random variables $X$ and $Y$ in $L^2(\Omega;\RR^d)$.
From our discussion of the construction of $\partial u$, notice that, for each $\mu$, $\partial u(\mu)(\,\cdot\,)$ is only uniquely defined $\mu$-almost everywhere. The following lemma (the proof of which is deferred to Subsection \ref{subse:proof:le:5})
then says that, in the current framework, there is a Lipschitz continuous version of $\partial u(\mu)(\,\cdot\,)$:
\begin{lemma}
\label{le:5}
Given a family of Borel-measurable mappings $(v(\mu)(\, \cdot \,) : \RR^d \rightarrow \RR^d)_{\mu \in {\mathcal P}_{2}(\RR^d)}$ indexed by the probability measures of order $2$ on $\RR^d$, assume that there exists a constant $C$ such that, for any square integrable random variables 
$\xi$ and $\xi'$ in $L^2( \Omega;\RR^d)$, it holds 
\begin{equation}
\label{eq:28:2:5}
{\mathbb E} \bigl[ \vert v(\PP_{\xi})(\xi) - v(\PP_{\xi'})(\xi') \vert^2 \bigr] \leq C^2 {\mathbb E} \bigl[ \vert \xi - \xi' \vert^2
\bigr].
\end{equation}
Then, for each $\mu\in\cP_2(\RR^d)$, one can redefine $v(\mu)(\,\cdot\,)$ on a $\mu$-negligeable set in such a way that:
\begin{equation*}
\forall x,x' \in \RR^d, \quad \vert v(\mu)(x) - v(\mu)(x') \vert \leq C \vert x-x' \vert,
\end{equation*}
for the same $C$ as in \eqref{eq:28:2:5}. 
\end{lemma}

By \eqref{fo:Lip}, we can use Lemma \ref{le:5} in order to define $\partial u(\mu)(x)$ for every $\mu$ and every $x$ while preserving the Lipschitz property in the variable $x$.
From now on, we shall use this version of $\partial  u$. So, if $\mu,\nu\in\cP_2(\RR^d)$ and $X$ and $Y$ are random variables such that $\PP_X=\mu$ and $\PP_Y=\nu$, we have:
\begin{equation*}
\begin{split}
\EE \bigl[ \vert \partial u(\mu)(X) - \partial u(\nu)(X) \vert^2 \bigr]\le &
2\bigg(\EE \bigl[ \vert \partial u(\mu)(X) - \partial u(\nu)(Y) \vert^2 \bigr] 
+\EE \bigl[ \vert \partial u(\nu)(Y) - \partial u(\nu)(X) \vert^2 \bigr]\bigg)
\\
\leq& 4C^2 \EE \bigl[ \vert Y - X \vert^2 \bigr],
\end{split}
\end{equation*}
where we used the Lipschitz property \eqref{fo:Lip} of the derivative together with the result of  Lemma \ref{le:5} 
applied to the function $\partial u(\nu)$.
Now, taking the infimum over all the couplings $(X,Y)$ with marginals $\mu$ and $\nu$, we obtain
\begin{equation*}
\inf_{X,\PP_X=\mu} \EE \bigl[ \vert \partial u(\mu)(X) - \partial u(\nu)(X) \vert^2 \bigr]\le 4C^2 W_{2}(\PP_{X},\PP_{Y})^2
\end{equation*}
and since the left-hand side depends only upon $\mu$ and not on $X$ as long as $\PP_X=\mu$, we get
\begin{equation}
\label{fo:est}
\EE \bigl[ \vert \partial u(\mu)(X) - \partial u(\nu)(X) \vert^2 \bigr]\leq 4C^2 W_{2}(\mu,\nu)^2.
\end{equation}
We will use the following consequence of this estimate:
 
\begin{proposition}
\label{prop:14:3:1}
Let $u$  be a differentiable function on ${\mathcal P}_{2}(\RR^d)$ with a Lipschitz derivative, and let $\mu\in\cP_2(\RR^d)$, and  $\u x=(x_{1},\dots,x_{N})\in (\RR^d)^N$ and $\u y=(y_{1},\dots,y_{N})\in (\RR^d)^N$. Then, with the same notation as in the statement of Proposition \ref{pr:diff_empir},  we have:
\begin{equation*}
\partial \bar{u}^N(\u x)\cdot(\u y - \u x) = \frac{1}{N} \sum_{i=1}^N 
 \partial u (\mu)(x_{i}) (y_{i} - x_{i})
+ {\mathcal O} \biggl[ W_{2}(\bar{\mu}_{N},\mu) \biggl( N^{-1} \sum_{i=1}^N \vert x_{i} - y_{i} \vert^2 \biggr)^{1/2} \biggr], 
\end{equation*}
the dot product being here the usual Euclidean inner product
and ${\mathcal O}$ standing for the Landau notation. 
\end{proposition}

\begin{proof}
Using Proposition \ref{pr:diff_empir}, we get:
\begin{equation*}
\begin{split}
\partial \bar{u}^N(\u x)\cdot(\u y - \u x)
&=\sum_{i=1}^N \partial _{x_i}\bar{u}^N(\u x)  (y_i-x_i)
\\
&=\frac1N\sum_{i=1}^N\partial u(\bar{\mu}^N)(x_i)  (y_i-x_i)
\\
&=\frac1N\sum_{i=1}^N \partial u(\mu)(x_i)  (y_i-x_i)
+\frac1N\sum_{i=1}^N[\partial u(\bar{\mu}^N)(x_i) - \partial u(\mu)(x_i)]  (y_i-x_i).
\end{split}
\end{equation*}
Now, by Cauchy-Schwarz' inequality,
\begin{equation*}
\begin{split}
&\bigg|\frac1N\sum_{i=1}^N[ \partial u(\bar{\mu}^N)(x_i) - \partial u(\mu)(x_i)]  (y_i-x_i) \bigg| 
\\
&\le \bigg(\frac1N\sum_{i=1}^N | \partial u(\bar{\mu}^N)(x_i) - \partial u(\mu)(x_i)|^2\bigg)^{1/2}
\bigg(\frac1N\sum_{i=1}^N|y_i-x_i|^2\bigg)^{1/2}
\\
&= \big({\t \EE}\bigl[ | \partial u(\bar{\mu}^N)(x_\vartheta) - \partial u(\mu)(x_\vartheta)|^2 \bigr]\big)^{1/2}
\bigg(\frac1N\sum_{i=1}^N|y_i-x_i|^2\bigg)^{1/2}
\\
&\le 2CW_2(\bar{\mu}^N,\mu)\bigg(\frac1N\sum_{i=1}^N|y_i-x_i|^2\bigg)^{1/2},
\end{split}
\end{equation*}
if we use the same notation for $\vartheta$ as in the proof of Proposition 
\ref{pr:diff_empir}
 and apply the estimate \eqref{fo:est} with $X=x_\vartheta$, $\mu=\bar{\mu}^N$ and $\nu=\mu$.
\end{proof}

\begin{remark}
\label{re:convergence}
We shall use the estimate of Proposition \ref{prop:14:3:1} when $x_i=X_i$ and the $X_i$'s are independent $\RR^d$-valued random variables with common distribution $\mu$. Whenever $\mu \in {\mathcal P}_{2}(\RR^d)$, the law of large numbers ensures that the Wasserstein distance between $\mu$ and the empirical measure $\bar{\mu}^N$ tends to $0$ a.s., that is 
\begin{equation*}
\PP \bigl( \lim_{n \rightarrow + \infty} W_{2}(\bar{\mu}^N,\mu) = 0 \bigr) = 1, 
\end{equation*}
see for example Section 10 in \cite{RachevRuschendorf}. Since we can find a constant $C >0$, independent of $N$, such that 
\begin{equation*}
W_{2}^2(\bar{\mu}^N,\mu) \leq C \bigl( 1 + \frac{1}{N} \sum_{i=1}^N \vert X_{i} \vert^2 \bigr),
\end{equation*}
we deduce from the law of large numbers again that the family $(W_{2}^2(\bar{\mu}^N,\mu))_{N \geq 1}$ is uniformly integrable, so that the convergence to $0$ also holds in the $L^2$ sense:
\begin{equation}
\label{fo:wassertsein empirical}
\lim_{n \rightarrow + \infty} \EE \bigl[ W_{2}^2(\bar{\mu}^N,\mu) \bigr] = 0.  
\end{equation}
Whenever $\int_{\RR^d}|x|^{d+5}\mu(dx)<\infty$, the rate of convergence can be specified. We indeed have the following standard estimate on the Wasserstein distance between $\mu$ and the empirical measure $\bar{\mu}^N$: 
\begin{equation}
\label{fo:horowitz}
{\mathbb E} \bigl[ W^2_{2}(\bar{\mu}^N,\mu) \bigr] \leq C N^{-2/(d+4)},
\end{equation}
for some constant $C>0$, 
see for example Section 10 in \cite{RachevRuschendorf}. 
Proposition \ref{prop:14:3:1} then says that, when $N$ is large, the gradient of $\bar{u}^N$ at the empirical sample $(X_{i})_{1 \leq i \leq N}$ is close to the sample 
$(\partial u(\mu)(X_{i}))_{1 \leq i \leq N}$, the accuracy of the approximation being specified in the $L^2({\Omega})$ norm by \eqref{fo:horowitz} when $\mu$ is sufficiently integrable.
\end{remark}

\subsection{Joint Differentiability and Convexity}
\label{subse:joint}
{\ }

\textit{Joint Differentiability.}
Below, we often consider functions $g :  \RR^n \times {\mathcal P}_{2}(\RR^d) \ni (x,\mu) \rightarrow  g(x,\mu) \in \RR$ depending on both an $n$-dimensional $x$  and a probability measure $\mu$. Joint differentiability is then defined according to the same procedure: $g$ is said to be jointly differentiable if the \emph{lifting} $\tilde{g} : \RR^n \times L^2(\t \Omega;\RR^d) \ni (x,\t X) \mapsto 
g(x,\t \PP_{\t X})$ is jointly differentiable. In such a case, we can define the partial derivatives in $x$ and $\mu$: they read
$\RR^d \times {\mathcal P}_{2}(\RR^d) \ni (x,\mu) \mapsto \partial_{x} g(x,\mu)$ and $ 
\RR^d \times {\mathcal P}_{2}(\RR^d) \ni (x,\mu) \mapsto \partial_{\mu} g(x,\mu)(\cdot) \in L^2(\RR^d,\mu)$ respectively. The partial Fr\'echet derivative of $\tilde{g}$ in the direction $\tX$ thus reads $L^2(\t \Omega;\RR^d) \ni (x,\tX) \mapsto D_{\tX} \t g(x,\tX) =
\partial_{\mu} g(x,\t \PP_{\t X})(\t X) \in L^2(\t \Omega;\RR^d)$.  

We often use the fact that joint continuous differentiability in the two arguments is equivalent with partial differentiability in each of the two arguments and joint continuity of the partial derivatives. Here, the joint continuity of $\partial_{x} g$ is understood as the joint continuity with respect to the Euclidean distance on $\RR^n$ and the Wasserstein distance on ${\mathcal P}_{2}(\RR^d)$. The joint continuity of $\partial_{\mu} g$ is understood as the joint continuity of the mapping
$(x,\tX) \mapsto \partial_{\mu} g(x,\t \PP_{\t X})(\t X)$ from $\RR^n \times L^2(\t \Omega;\RR^d)$ into $L^2(\t \Omega;\RR^d)$.   

When the partial derivatives of $g$ are assumed to be Lipschitz-continuous, we can benefit 
from Lemma \ref{le:5}. It says that, for any $(x,\mu)$, the \emph{representation} $\RR^d \ni x' \mapsto \partial_{\mu} g(x,\mu)(x')$ makes sense as a Lipschitz function in $x'$ and that an appropriate version of \eqref{fo:est} holds true. 
\vspace{4pt}

\textit{Convex Functions of Measures.}
We define a notion of convexity associated with this notion of differentiability.
A function $g$ on $\cP_2(\RR^d)$ which is differentiable in the above sense is said to be convex if for every $\mu$ and $\mu'$ in 
$\cP_2(\RR^d)$ we have
\begin{equation}
\label{fo:convexity:0}
g(\mu')-g(\mu) - \t\EE[\partial_\mu g( \mu)(\t X)\cdot (\t{X'}- \t X)]\ge 0
\end{equation}
whenever $\t X$ and $\t{X'}$ are square integrable random variables with distributions $\mu$ and $\mu'$ respectively.
Examples are given in Subsection \ref{subs:examples}.

More generally, a function $g$ on $\RR^n \times\cP_2(\RR^d)$ which is jointly differentiable in the above sense is said to be convex if for every $(x,\mu)$ and $(x',\mu')$ in $\RR^n \times\cP_2(\RR^d)$ we have
\begin{equation}
\label{fo:convexity}
g(x',\mu')-g(x,\mu)-\partial_xg(x,\mu)\cdot (x'- x) - \t\EE[\partial_\mu g( x, \mu)(\t X)\cdot (\t{X'}- \t X)]\ge 0
\end{equation}
whenever $\t X$ and $\t{X'}$ are square integrable random variables with distributions $\mu$ and $\mu'$ respectively.

\subsection{Proof of Lemma \ref{le:5}}
\label{subse:proof:le:5}

\textit{First Step.} 
We first consider the case $v$ bounded and assume that $\mu$ has a strictly positive continuous density $p$ on the whole $\RR^d$, $p$ and its derivatives being of exponential decay at the infinity. 
We then claim that there exists a continuously differentiable one-to-one function from $(0,1)^d$ onto  $\RR^d$ such that, whenever 
$\eta_{1},\dots,\eta_{d}$ are $d$ independent random variables, each of them being uniformly distributed on $(0,1)$,
 $U(\eta_{1},\dots,\eta_{d})$ has distribution $\mu$. It satisfies for any $(z_{1},\dots,z_{d}) \in (0,1)^d$
 \begin{equation*}
 \frac{\partial U_{i}}{\partial z_{i}}(z_{1},\dots,z_{d}) \not = 0, \quad
  \frac{\partial U_{j}}{\partial z_{i}}(z_{1},\dots,z_{d}) = 0, \quad 1 \leq i < j \leq d. 
 \end{equation*}
 The result is well-known when $d=1$. In such a case, $U$ is the inverse of the cumulative distribution function of $\mu$. In higher dimension, $U$ can be constructed by an induction argument on the dimension. Assume indeed that some $\hat{U}$ has been constructed for the first marginal distribution $\hat{\mu}$ of $\mu$ on $\RR^{d-1}$, that is for the push-forward of $\mu$ by the projection  
 mapping $\RR^d \ni (x_{1},\dots,x_{d}) \mapsto (x_{1},\dots,x_{d-1})$.  Given $(x_{1},\dots,x_{d-1}) \in \RR^{d-1}$, we then denote by $p(\cdot \vert x_{1},\dots,x_{d-1})$ the conditional density of $\mu$ given the $d-1$ first coordinates:
 \begin{equation*}
 p(x_{d} \vert x_{1},\dots,x_{d-1}) = \frac{p(x_{1},\dots,x_{d})}{\hat{p}(x_{1},\dots,x_{{d-1}})}, \quad 
 x_{1},\dots,x_{d-1} \in \RR^{d-1},
 \end{equation*}
 where $\hat{p}$ denotes the density of $\hat{\mu}$ (which is continuously differentiable and positive). We then denote by 
 $(0,1) \ni z_{d} \mapsto U_{d}(z_{d}\vert x_{1},\dots,x_{d-1})$ the inverse of the cumulative distribution function of the law of density $p(\, \cdot \, \vert x_{1},\dots,x_{d-1})$. It satisfies
 \begin{equation*}
 F_{d}\bigl(U_{d}(z_{d}\vert x_{1},\dots,x_{d-1}) \vert x_{1},\dots,x_{d-1}\bigr) = z_{d},
 \end{equation*} 
 with
 \begin{equation*}
 F_{d}(x_{d}\vert x_{1},\dots,x_{d-1}) = \int_{-\infty}^{x_{d}} p(y \vert x_{1},\dots,x_{d-1}) dy,
 \end{equation*}
 which is continuously differentiable in $(x_{1},\dots,x_{d})$ (using the exponential decay of the density at the infinity). 
By the implicit function theorem, the mapping $\RR^{d-1} \times (0,1) \ni (x_{1},\dots,x_{d-1},z_{d}) \mapsto 
U_{d}(z_{d} \vert x_{1},\dots,x_{d-1})$ is continuously differentiable. The partial derivative with respect to $z_{d}$ is given by 
\begin{equation*}
 \frac{\partial U_{d}}{\partial z_{d}}(z_{d} \vert x_{1},\dots,x_{d-1}) = \frac{1}{p( U_{d}(z_{d} \vert x_{1},\dots,x_{d-1}) \vert x_{1},\dots,x_{d-1})},
 \end{equation*}
 which is non-zero. We now let 
 \begin{equation*}
 U(z_{1},\dots,z_{d}) = \bigl( \hat{U}(z_{1},\dots,z_{d-1}), U_{d}(z_{d} \vert \hat{U}(z_{1},\dots,z_{d-1})) \bigr), \quad 
 z_{1},\dots,z_{d} \in (0,1)^d. 
 \end{equation*}
By construction, $U(\eta_{1},\dots,\eta_{d})$ has distribution $\mu$: $(\eta_{1},\dots,\eta_{d-1})$ has distribution $\hat{\mu}$ and the conditional law of $U_{d}(\eta_{1},\dots,\eta_{d})$ given $\eta_{1},\dots,\eta_{d-1}$ is the conditional law of $\mu$ given the $d-1$ first coordinates. It satisfies $[\partial U_{d}/\partial z_{d}](z_{1},\dots,z_{d}) > 0$
and $[\partial U_{i}/\partial z_{d}](z_{1},\dots,z_{d}) = 0$ for $i < d$. In particular, since $\hat{U}$ is assumed (by induction) to be injective and
$[\partial U_{d}/\partial z_{d}](z_{1},\dots,z_{d}) > 0$, $U$ must be injective as well. As the Jacobian matrix of $U$ is triangular with non-zero elements on the diagonal, it is invertible. By the global inversion theorem, $U$ is a diffeomorphism: the range of $U$ is the support of $\mu$, that is $\RR^d$. This proves that $U$ is one-to-one from $(0,1)^d$ onto $\RR^d$. 
\vspace{2pt}

\textit{Second Step.} We still consider the case $v$ bounded and assume that $\mu$ has a strictly positive continuous density $p$ on the whole $\RR^d$, $p$ and its derivatives being of exponential decay at the infinity. We will use the mapping $U$ constructed in the first step. 
For three random variables $\xi$, $\xi'$ and $G$ in $L^2(\Omega;\RR^d)$, the pair $(\xi,\xi')$ being independent of $G$, the random variables 
$\xi$ and $\xi'$ having the same distribution, and $G$ being ${\mathcal N}_{d}(0,I_{d})$ normally distributed, \eqref{eq:28:2:5} implies that, for any integer $n \geq 1$
\begin{equation*}
\EE \bigl[ \vert v\bigl(\xi+ n^{-1} G,\PP_{\xi + n^{-1} G}\bigr)
- v\bigl(\xi'+ n^{-1} G,\PP_{\xi + n^{-1} G}\bigr) \vert^2 \bigr] \leq C^2 \EE \bigl[ \vert \xi - \xi' \vert^2 \bigr].
\end{equation*}
In particular, setting 
\begin{equation*}
v_{n}(x) = \frac{n^d}{(2 \pi)^{d/2}}\int_{\RR^d} v\bigl(y,\PP_{\xi + n^{-1} G}\bigr) \exp \bigl( -n^2 \frac{ \vert x - y \vert^2}{2} \bigr) dy,
\end{equation*}
we have
\begin{equation}
\label{eq:25:3:2}
\EE \bigl[ \vert v_{n}(\xi) - v_{n}(\xi') \vert^2 \bigr] \leq C^2 \EE \bigl[ \vert \xi - \xi' \vert^2 \bigr].
\end{equation}
Notice that $v_{n}$ is infinitely differentiable with bounded derivatives. 

We now choose a specific coupling for $\xi$ and $\xi'$. Indeed, we know that for any $ \eta=(\eta_1,\dots,\eta_{d})$ and 
$\eta'=(\eta_{1}',\dots,\eta_{d}')$, with uniform distribution on $(0,1)^d$, $U( \eta)$ and $U(\eta')$ have the same distribution as $\xi$. Without any loss of generality, we then assume that the probability space $(\Omega,{\mathcal F},\PP)$ is given by 
$(0,1)^d \times \RR^d$ endowed with its Borel $\sigma$-algebra and the product of the Lebesgue measure on $(0,1)^d$ and 
of the Gaussian measure ${\mathcal N}_{d}(0,I_{d})$. The random variables $\eta$ and $G$ are then chosen as
the canonical mappings  
$\eta : (0,1)^d \times \RR^d \ni (z,y) \mapsto z$ and $G : (0,1)^d \times \RR^d \ni (z,y) \mapsto y$.

We then define $\eta'$ as a function of the variable $z \in (0,1)^d$ only. For a given $z^{0}=(z^{0}_{1},\dots,z^{0}_{d}) \in (0,1)^d$ and for $h$ small enough so that the open ball $B(z^0,h)$ of center $z^{0}$ and radius $h$ is included in $(0,1)^d$, we let:
\begin{equation*}
\eta'(z)=
\left\{ \begin{array}{ll}
z^0 + (  z^{0}_{d} - z_{d}) e_{d} &\quad \textrm{on} \ B(z^0,h),
\\
z, &\quad \textrm{outside} \ B(z^0,h),
\end{array}
\right.
\end{equation*}
where $e_{d}$ is the $d$th vector of the canonical basis. 
We rewrite \eqref{eq:25:3:2} as:
\begin{equation*}
\int_{(0,1)^d} \bigl\vert v_{n}\bigl(U(\eta(z))\bigr) - v_{n}\bigl(U(\eta'(z))\bigr) \bigr\vert^2 dz \leq 
C^2 \int_{(0,1)^d} \bigl\vert U(\eta(z)) -U(\eta'(z)) \bigr\vert^2 dz,
\end{equation*}
or equivalently:
\begin{equation}
\label{fo:50}
\begin{split}
&\int_{\vert r \vert < h} \bigl\vert v_{n}\bigl[U \bigl(z^{0} + r - 2 r_{d} e_{d} \bigr) \bigr]
 - v_{n}\bigl(U(z^{0} + r)\bigr) \bigr\vert^2 dr 
\\
&\hspace{15pt} \leq 
C^2 \int_{\vert r \vert < h} \bigl\vert U\bigl( z^{0}+ r - 2 r_{d} e_{d} \bigr)
-U(z^{0}+r) \bigr\vert^2 dr.
\end{split}
\end{equation}
Since $U$ is continuously differentiable, we have 
\begin{equation*}
v_{n}(U(z^{0}+r)) = v_{n}(U(z^{0})) + \partial v_{n}(U(z^{0}))  \cdot \bigl[ \partial U(z^0) \cdot r \bigr] + o(r),
\end{equation*}
where $\partial U(z^0)$ is a $d \times d$ matrix. 
We deduce
\begin{equation*}
\begin{split}
v_{n}\bigl[ U \bigl(z^{0}+ r - 2r_{d} e_{d}
\bigr)\bigr] - v_{n}(U(z^{0}+r))  &= - 2 \sum_{i=1}^d 
\frac{\partial v_{n}}{\partial x_{i}}(U(z^{0}))
\frac{\partial U_{i}}{\partial z_{d}}(z^{0}) r_{d}+ o(r)
\\
&= - 2 \frac{\partial v_{n}}{\partial x_{d}}(U(z^{0}))
\frac{\partial U_{d}}{\partial z_{d}}(z^{0}) r_{d}+ o(r),
\end{split}
\end{equation*}
since $\partial U_{i}/\partial z_{d} =0$ for $i \not = d$,
and
\begin{equation}
\label{fo:51}
\begin{split}
&\int_{\vert r \vert <h} \bigl\vert v_{n}\bigl[U \bigl(z^{0}+r - 2r_{d} e_{d}\bigr) \bigr] - v_{n}(U(z^{0}-r)) \bigr\vert^2 dr
\\
&= 4 \bigl\vert
\frac{\partial v_{n}}{\partial x_{d}}(U(z^{0})) 
\frac{\partial U_{d}}{\partial z_{d}}(z^{0}) \bigr\vert^2 \int_{\vert r \vert <h} r_{d}^2 dr + o(h^3).
\end{split}
\end{equation}
Similarly, 
\begin{equation}
\label{fo:52}
\int_{\vert r \vert < h} \bigl\vert U(z^{0}+r - 2r_{d} e_{d}) - U(z^{0}+r) \bigr\vert^2 dr
= 4 \bigl\vert 
\frac{\partial U_{d}}{\partial z_{d}}(z^{0}) \vert^2 \int_{\vert r \vert <h} r_{d}^2 dr + o(h^3).
\end{equation}
and putting together \eqref{fo:50},  \eqref{fo:51}, and \eqref{fo:52}, we obtain
\begin{equation*}
\bigl\vert \frac{\partial v_{n}}{\partial x_{d}}(U(z^{0})) \frac{\partial U_{d}}{\partial z_{d}}(z^{0}) \vert^2
 \leq 
C^2 
\bigl\vert \frac{\partial U_{d}}{\partial z_{d}}(z^{0}) \vert^2.
\end{equation*}
Since $[\partial U_{d}/\partial z_{d}](z^{0})$ is different from zero, we deduce that 
\begin{equation*}
\bigl\vert \frac{\partial v_{n}}{\partial x_{d}}(U(z^{0}))  \vert^2
 \leq C^2,
\end{equation*}
and since $U$ is 
a one-to-one mapping from $(0,1)^d$ onto $\RR^d$, and $z^0\in(0,1)^d$ is arbitrary, we conclude that
$\vert [\partial v_{n}/\partial x_{d}](x) \vert \leq C$, for any $x \in \RR^d$. By changing the basis used for the construction of $U$ (we used the canonical one but we could use any orthonormal basis), we have $\vert \nabla v_{n}(x) e \vert \leq C$ for any $x,e \in \RR^d$ with $\vert e \vert =1$. This proves that the functions 
$(v_{n})_{n \geq 1}$ are uniformly bounded and $C$-Lipschitz continuous. We then denote by $\hat{v}$
the limit of a subsequence converging for the topology of uniform convergence on compact subsets. For simplicity, we keep the index $n$ to denote the subsequence.  Assumption \eqref{eq:28:2:5} implies: 
\begin{equation*}
\EE \bigl[ \vert v_{n}(\xi) - v(\xi,\PP_{\xi}) \vert^2 \bigr] \leq 
\EE \bigl[ \vert v(\xi+n^{-1}G,\PP_{\xi+n^{-1}G}) - v(\xi,\PP_{\xi}) \vert^2 \bigr] 
\leq C^2 n^{-2},
\end{equation*}
and taking the limit $n\to+\infty$, we deduce that $\hat{v}$ and $v(\,\cdot\,,\PP_{\xi})$ coincide $\PP_{\xi}$ almost everywhere. This completes the proof when 
$v$ is bounded and $\xi$ has a continuous positive density $p$, $p$ and its derivatives being of exponential decay at the infinity. 
\vspace{2pt}

\textit{Third Step.} When $v$ is bounded and $\xi$ is bounded and has a general distribution, we approximate $\xi$ by $\xi+n^{-1} G$ again. Then, 
$\xi + n^{-1} G$ has a positive continuous density, the density and its derivatives being of Gaussian decay at the infinity, so that, by the second step, the function $\RR^d \ni x \mapsto v(x,\PP_{\xi+n^{-1}G})$ 
can be assumed to be $C$-Lipschitz continuous for each $n\ge 1$. Extracting a converging subsequence and passing to the limit as above, we deduce that $v(\cdot,\PP_{\xi})$ admits a $C$-Lipschitz continuous version. 

When $v$ is bounded but $\xi$ is not bounded, we approximate $\xi$ by its orthogonal projection on the ball of center $0$ and radius $n$. We then complete the proof in a similar way. 

Finally when $v$ is not bounded, we approximate $v$ by $(\psi_{n}(v))_{n \geq 1}$ where, for each $n \geq 1$, 
$\psi_{n}$ is a bounded smooth function from $\RR$ into itself such that $\psi_{n}(r)=r$ for $r\in [-n,n]$ and $\vert [d\psi_{n}/dr](r) \vert \leq 1$ for all $r \in \RR$. Then, for each $n \geq 1$, there exists a $C$-Lipschitz continuous 
version of $\psi_{n}(v(\cdot,\PP_{\xi}))$. Choosing some $x_{0} \in \RR^d$ such that $\vert v(x_{0},\PP_{\xi}) \vert < + \infty$,  the sequence $\psi_{n}(v(x_{0},\PP_{\xi}))$ is bounded so that the sequence of functions $(\psi_{n}(v( \, \cdot \, ,\PP_{\xi})))_{n \geq 1}$ is uniformly bounded and continuous on compact subsets. Extracting a converging subsequence, we complete the proof in the same way as before. 

\subsection{The Hamiltonian and the Dual Equations}
The Hamiltonian of the stochastic optimization problem is defined as the function $H$ given by
\begin{equation}
\label{fo:hamiltonian}
H(t,x,\mu,y,z,\alpha)=b(t,x,\mu,\alpha)\cdot y +\sigma(t,x,\mu,\alpha)\cdot z + f(t,x,\mu,\alpha)
\end{equation}
where the dot notation stands here for the inner product in an Euclidean space.
Because we need to compute derivatives of $H$ with respect to its variable $\mu$, we consider the lifting $\t H$ defined by
\begin{equation}
\label{fo:liftedhamiltonian}
\t H(t,x,\t X,y,z,\alpha)=H(t,x,\mu,y,z,\alpha)
\end{equation}
for any random variable $\t X$ with distribution $\mu$, and we shall denote by $\partial_\mu H(t,x,\mu_0,y,z,\alpha)$ the derivative  with respect to $\mu$ computed at $\mu_0$ (as defined above) whenever all the other variables $t$, $x$, $y$, $z$ and $\alpha$ are held fixed. We recall that $\partial_\mu H(t,x,\mu_0,y,z,\alpha)$ is an element of $L^2(\RR^d,\mu_{0})$ and that we identify it with a function $\partial_\mu H(t,x,\mu_0,y,z,\alpha)(\, \cdot \,) : \RR^d \ni \tilde{x} \mapsto \partial_\mu H(t,x,\mu_0,y,z,\alpha)(\t x)$. It satisfies
$D \t H(t,x,\t X,y,z,\alpha) = \partial_\mu H(t,x,\mu_0,y,z,\alpha)(\t X)$  almost-surely under $\t \PP$. 
\begin{definition}
\label{de:adjoint}
In addition to (A1-2), assume that the coefficients $b,\sigma,f$ and $g$ are differentiable with respect to
$x$ and $\mu$. Then, given an admissible control $\alpha=(\alpha_t)_{0\le t\le T}\in\bA$, we denote by $ X=  X^{\alpha}$ the corresponding controlled state process. Whenever
\begin{equation}
\label{eq:L2:a}
\EE \int_{0}^T \bigl\{ \vert \partial_{x} f(t,X_{t},\PP_{X_{t}},\alpha_{t}) \vert^2  + 
\t \EE \bigl[ \vert \partial_{\mu} f(t,X_{t},\PP_{X_{t}},\alpha_{t})(\t X_{t}) \vert^2 \bigr]\bigr\} dt
< + \infty,
\end{equation}
and 
\begin{equation}
\label{eq:L2:b}
\EE \bigl\{ \vert \partial_{x} g(X_{T},\PP_{X_{T}}) \vert^2 + \t \EE \bigl[ \vert \partial_{\mu} g(X_{T},\PP_{X_{T}})(\t X_{T}) \vert^2 \bigr] \bigr\}
< + \infty,
\end{equation}
we call adjoint processes of $X$ any couple  $((Y_t)_{0\le t\le T},(Z_{t})_{0 \leq t \leq T})$ of progressively measurable stochastic processes in $\HH^{2,d} \times \HH^{2,d\times m}$
 satisfying the equation (which we call the adjoint equation):
\begin{equation}
\label{fo:adjoint}
\begin{cases}
&dY_t=-\partial_xH(t,X_t,\PP_{X_t},Y_t,Z_t,\alpha_t)dt + Z_t dW_t\\
&\phantom{?????????????????????????}-\t\EE[\partial_\mu \t H(t,\t X_t,\PP_{X_t},\t Y_t,\t Z_t,\t\alpha_t)(X_t)]dt\\
& Y_T= \partial_xg(X_T,\PP_{X_T})+\t\EE[\partial_\mu g(\t X_T,\PP_{X_T})(X_T)]
\end{cases}
\end{equation}
where $(\t X,\t Y,\t Z,\t\alpha)$ is an independent copy of $(X,Y,Z,\alpha)$ defined on $L^2(\tilde{\Omega},\t {\mathcal F},\t \PP)$
 and $\t\EE$ denotes the expectation on $(\tilde{\Omega},\t {\mathcal F},\t \PP)$.
\end{definition}
Notice that 
$\EE[\partial_\mu \t H(t,\t X_t,\PP_{X_t},\t Y_t,\t Z_t,\t\alpha_t)(X_t)]$ is a function of the random variable $X_t$ as it stands for 
$\EE[\partial_\mu \t H(t,\t X_t,\PP_{X_t},\t Y_t,\t Z_t,\t\alpha_t)(x)]_{\vert x= X_{t}}$ (and similarly for 
$\t\EE[\partial_\mu g(\t X_T,\PP_{X_T})(X_T)]$). Notice that, when $b$, $\sigma$, $f$ and $g$ do not depend upon the marginal distributions of the controlled state process,  the extra terms appearing in the adjoint equation and its terminal condition disappear and this equation coincides with the classical adjoint equation of stochastic control. 

\vskip 2pt
Using the appropriate interpretation of the symbol $\odot$ as explained in Section \ref{se:mkv} and extending this notation to 
derivatives of the form $\partial_{\mu} h(\mu)(x) \odot p = (\partial_{\mu} [ h(\mu) \cdot p])(x)$, the adjoint equation rewrites
\begin{equation}
\label{fo:adjoint'}
\begin{split}
dY_t&=-\bigl[\partial_xb(t,X_t,\PP_{X_t},\alpha_t)\odot Y_t+\partial_x\sigma(t,X_t,\PP_{X_t},\alpha_t)\odot Z_t+\partial_xf(t,X_t,\PP_{X_t},\alpha_t) \bigr]dt 
\\
&\hspace{250pt}+Z_t dW_t
\\
&\hspace{15pt}
-\t\EE\bigl[\partial_\mu b(t,\t X_t,\PP_{X_t},\t\alpha_t)(X_t)\odot \t Y_t
+\partial_\mu \sigma(t,\t X_t,\PP_{X_t},\t\alpha_t)(X_t)\odot \t Z_t
\\
&\hspace{250pt}+\partial_\mu f(t,\t X_t,\PP_{X_t},\t\alpha_t)(X_t) \bigr]dt,
\end{split}
\end{equation}
with the terminal condition $Y_T= \partial_xg(X_T,\PP_{X_T})+\t\EE[\partial_\mu g(\t X_T,\PP_{X_T})( X_T)]$.
Notice that $\partial_{x} b$ and $\partial_{x} \sigma$ are bounded since $b$ and $\sigma$ are assumed to be $c$-Lipschitz continuous in the variable $x$, see (A2). Notice also that $\EE [ \vert \partial_{\mu} b(t,\t X_t,\PP_{X_t},\t\alpha_t)(X_t) \vert^2]^{1/2}$ and
$\EE [ \vert \partial_{\mu} \sigma(t,\t X_t,\PP_{X_t},\t\alpha_t)(X_t) \vert^2]^{1/2}$ are also bounded by $c$ since 
$b$ and $\sigma$ are assumed to be $c$-Lipschitz continuous in the variable $\mu$ with respect to the 2-Wasserstein distance. It is indeed plain to check that, given a differentiable function $h : {\mathcal P}_{2}(\RR^d) \rightarrow \RR$, the notion of differentiability being defined as above, it holds 
${\mathbb E}[ \vert \partial_{\mu} h(X) \vert^2]^{1/2} \leq c$, for any $\mu \in {\mathcal P}_{2}(\RR^d)$ and 
any random variable $X$ having $\mu$ as distribution, when $h$ is $c$-Lipschitz continuous in $\mu$ with respect to the 2-Wasserstein distance. 

Notice finally that, given an admissible control $\alpha\in\bA$ and the corresponding controlled state process $ X= X^{\alpha}$, despite  the conditions (\ref{eq:L2:a}--\ref{eq:L2:b}) and despite the fact that the first part of the equation appears to be linear in the unknown processes $Y_t$ and $Z_t$, existence and uniqueness of a solution $( Y,  Z)$ of the adjoint equation is not provided by standard results on Backward Stochastic Differential Equations (BSDEs) as the distributions of the solution processes (more precisely their joint distributions with the control and state processes $\alpha$ and $ X$) appear in the coefficients of the equation.
However, a slight modification of the original existence and uniqueness result of Pardoux and Peng \cite{PardouxPeng1990} shows that existence and uniqueness still hold in our more general setting. The main lines of the proof are given in \cite{BuckdahnDjehicheLiPeng},  Proposition 3.1 and Lemma 3.1. However, Lemma 3.1 in \cite{BuckdahnDjehicheLiPeng} doesn't apply directly  since the coefficients 
$(\partial_\mu b(t,\t X_t,\PP_{X_t},\t\alpha_t)(X_t)\odot \t Y_t)_{0 \leq t \leq T}$ and 
$(\partial_\mu \sigma(t,\t X_t,\PP_{X_t},\t\alpha_t)(X_t)\odot \t Z_t)_{0 \leq t \leq T}$ are not Lipschitz continuous in 
$\tilde{Y}$ and $\tilde{Z}$ uniformly in the randomness, see Condition (C1) in \cite{BuckdahnDjehicheLiPeng}. Actually, a careful inspection of the proof shows that the bounds
\begin{equation*}
\begin{split}
&{\mathbb E} \t \EE \bigl[ \vert \partial_\mu b(t,\t X_t,\PP_{X_t},\t\alpha_t)(X_t)\odot \t Y_t \vert^2 \bigr] \leq c' \EE \bigl[ \vert Y_{t} \vert^2 \bigr],
\\
&{\mathbb E} \t \EE \bigl[ \vert \partial_\mu \sigma(t,\t X_t,\PP_{X_t},\t\alpha_t)(X_t)\odot \t Z_t \vert^2 \bigr] \leq c' \EE \bigl[ \vert Z_{t} \vert^2 \bigr],
\end{split}
\end{equation*}
are sufficient to make the whole argument work and thus to prove existence and uniqueness of a solution $( Y, Z)$ satisfying
$$
\EE \biggl[ \sup_{0\le t\le T}|Y_t|^2 + \int_0^T|Z_t|^2dt \biggr] < +\infty.
$$

\section{\textbf{Pontryagin Principle for Optimality}}
\label{se:pontryagin}

In this section, we discuss sufficient and necessary conditions for optimality
when the Hamiltonian satisfies appropriate assumptions of convexity. These conditions will be specified next, 
depending upon the framework. For the time being, we detail the regularity properties that will be used throughout the section. 
 Referring to Subsection \ref{subse:joint} for definitions of joint differentiability, we assume:
\vspace{5pt}

(A3) The functions $b$, $\sigma$ and 
$f$ are differentiable with respect to $(x,\alpha)$, the mappings $(x,\mu,\alpha) \mapsto 
\partial_{x} (b,\sigma,f) (t,x,\mu,\alpha)$ and $(x,\mu,\alpha) \mapsto 
\partial_{\alpha} (b,\sigma,f) (t,x,\mu,\alpha)$ being continuous for any $t \in [0,T]$. The functions $b$, $\sigma$ and $f$ are also differentiable with respect to the variable $\mu$ in the sense given above, the mapping 
$\RR^d \times 
L^2(\Omega;\RR^d) 
\times A \ni (x,X,\alpha) \mapsto \partial_{\mu} (b,\sigma,f)(t,x,\PP_{X},\alpha)(X) \in L^2(\Omega;\RR^{d \times d} \times \RR^{(d \times m) \times d} \times \RR^d)$ being continuous for any 
$t \in [0,T]$. Similarly, the function $g$ is differentiable with respect to $x$, the mapping 
$(x,\mu) \mapsto \partial_{x} g(x,\mu)$ being continuous. The function $g$ is also differentiable with respect to the variable $\mu$, the mapping $\RR^d \times L^2(\Omega;\RR^d) \mapsto \partial_{\mu} g(x,\PP_{X})(X) \in L^2(\Omega;\RR^d)$ being continuous.
\vspace{2pt}

(A4) The coefficients $((b,\sigma,f)(t,0,\delta_{0},0))_{0 \leq t \leq T}$ are uniformly bounded. 
The partial derivatives $\partial_{x} b$ and $\partial_{\alpha} b$ are uniformly bounded and the norm of the mapping $x' \mapsto
\partial_{\mu}(b,\sigma)(t,x,\mu,\alpha)(x')$ in $L^2(\RR^d,\mu)$ 
is also uniformly bounded (i.e. uniformly in $(t,x,\mu,\alpha)$). There exists a constant $L$ such that, for any $R \geq 0$ and any 
$(t,x,\mu,\alpha)$ such that $\vert x \vert,\| \mu \|_{2},\vert \alpha \vert \leq R$, $\vert \partial_{x} (f,g)(t,x,\mu,\alpha) \vert$ 
and $\vert \partial_{\alpha} f(t,x,\mu,\alpha) \vert$ are bounded by $L(1+R)$ and the $L^2(\RR^d,\mu)$-norm of 
$x' \mapsto
\partial_{\mu}(f,g)(t,x,\mu,\alpha)(x')$ is bounded by $L(1+R)$. Here, we have used the notation
\begin{equation*}
\| \mu \|_{2}^2 = \int_{\RR^d} \vert x\vert^2 d\mu(x), \quad \mu \in {\mathcal P}_{2}(\RR^d). 
\end{equation*}

\subsection{\textbf{A Necessary Condition}}
\label{subsec:4:1}
We assume that the sets $A$ and $\bA$ of admissible controls are convex, we fix $\alpha\in\bA$, and as before, we denote by $ X= X^{\alpha}$ the corresponding controlled state process, namely the solution of \eqref{fo:mkvsde} with given initial condition $X_0=x_{0}$. Our first task is to compute the G\^ateaux derivative of the cost functional $J$ at $\alpha$ in all directions. In order to do so, we choose $\beta\in\HH^{2,k}$ such that $\alpha + \epsilon \beta \in \bA$ for $\epsilon >0$ small enough. We then compute the variation of $J$ at $\alpha$ in the direction of $\beta$ (think of $\beta$ as the difference between another element of $\bA$ and $\alpha$).

\vskip 2pt
Letting $(\theta_{t} = (X_{t},\PP_{X_{t}},\alpha_{t}))_{0 \leq t \leq T}$, we define the variation process $ V=(V_t)_{0\le t\le T}$ to be the solution of the equation
\begin{equation}
\label{fo:Voft}
dV_t=\bigl[\gamma_t \cdot V_t+\delta_t(\PP_{(X_t,V_t)})+\eta_t \bigr]dt+\bigl[\t\gamma_t \cdot V_t+\t\delta_t(\PP_{(X_t,V_t)})+\t\eta_t \bigr]dW_t,
\end{equation}
with $V_{0}=0$, where the coefficients $\gamma_t$, $\delta_t$, $\eta_t$, $\t\gamma_t$, $\t\delta_t$ and $\t\eta_t$ are defined as
\begin{eqnarray*}
&\gamma_t=\partial_xb(t,\theta_{t}), 
\quad \quad \t\gamma_t=\partial_x\sigma(t,\theta_{t}), \quad
\eta_t=\partial_\alpha b(t,\theta_{t}) \cdot \beta_t, \quad\text{ and }\quad \t\eta_t=\partial_\alpha \sigma(t,\theta_{t})
\cdot
\beta_t,
\end{eqnarray*}
which are progressively measurable bounded processes with values in $\RR^{d\times d}$, $\RR^{(d\times m) \times d}$, $\RR^{d}$, and $\RR^{d\times m}$ respectively
(the parentheses around $d\times m$ indicating that $\tilde{\gamma}_t \cdot u$ is seen as an element of $\RR^{d \times m}$ whenever $u \in \RR^d$), and
\begin{equation}
\label{fo:deltaoft}
\begin{split}
&\delta_t
=\t\EE \bigl[ \partial_\mu b(t,\theta_{t})(\t X_t)\cdot \t V_t \bigr]
=\t\EE \bigl[ \partial_\mu b(t,x,\PP_{X_t},\alpha)(\t X_t)\cdot \t V_t \bigr]_{\indices{x=X_t}{\alpha=\alpha_t}},
\\
&\t\delta_t
= \t\EE \bigl[
\partial_\mu \sigma(t,\theta_{t})(\t X_t)\cdot \t V_t\big]
=\t\EE \bigl[
\partial_\mu \sigma(t,x,\PP_{X_t},\alpha)(\t X_t)\cdot \t V_t\big]_{\indices{x=X_{t}}{\alpha=\alpha_{t}}},
\end{split}
\end{equation}
which are progressively measurable bounded processes with values in $\RR^{d}$ and $\RR^{d \times m}$ respectively and
where $(\t X_t,\t V_t)$ is an independent copy of $(X_t, V_t)$.  As expectations of functions of $(\t X_t, \t V_t)$, $\delta_t$ and $\t\delta_t$ 
depend upon the joint distribution of $X_t$ and $V_t$. In \eqref{fo:Voft} we wrote $\delta_t(\PP_{(X_t,V_t)})$ and $\t\delta_t(\PP_{(X_t,V_t)})$ 
in order to stress the dependence upon the joint distribution of $X_t$ and $V_t$. Even though we are dealing with possibly random coefficients, the existence and uniqueness of the variation process is guaranteed by Proposition 2.1 of \cite{JourdainMeleardWoyczynski} applied to the couple $( X, V)$ and the system formed
by \eqref{fo:mkvsde} and \eqref{fo:Voft}. Because of our assumption on the boundedness of the partial derivatives of the coefficients, $ V$  satisfies 
$\EE\sup_{0\le t\le T}\left| V_t
\right|^p<\infty$ for every finite $p\ge 1$.

\begin{lemma}
For each $\epsilon>0$ small enough, we denote by $\alpha^\epsilon$ the admissible control defined by $\alpha_t^\epsilon=\alpha_t+\epsilon\beta_t$, and by  $ X^\epsilon= X^{\alpha^\epsilon}$ the corresponding controlled state. We have:
\begin{equation}
\label{fo:variation}
\lim_{\epsilon\searrow 0}\EE\sup_{0\le t\le T}\left| \frac{X^\epsilon_t-X_t}{\epsilon}-V_t
\right|^2=0.
\end{equation}
\end{lemma}

\begin{proof}
For the purpose of this proof we set 
$\theta_{t}^{\epsilon}=(X_{t}^{\epsilon},\PP_{X_{t}^{\epsilon}},\alpha_{t}^{\epsilon})$ and 
$V^\epsilon_t= \epsilon^{-1}(X^\epsilon_t-X_t)-V_t$.
Notice that $V^\epsilon_0=0$ and that
\begin{equation}
\label{fo:Vepsilon}
\begin{split}
&dV^\epsilon_t
\\
&= \biggl[\frac{1}{\epsilon}\bigl[b(t,\theta_{t}^\epsilon)-b(t,\theta_{t})
\bigr]-\partial_x b(t,\theta_{t}) \cdot V_t-\partial_\alpha b(t,\theta_{t})\cdot \beta_t-\t\EE
\bigl[
\partial_\mu b(t,\theta_{t})(\t X_t)\cdot \t V_t \bigr] \biggr]dt\\
&\phantom{?}+\biggl[\frac{1}{\epsilon} \bigl[\sigma(t,\theta_{t}^\epsilon)-\sigma
(t,\theta_t) \bigr]-\partial_x \sigma(t,\theta_t) \cdot V_t
-\partial_\alpha \sigma(t,\theta_t) \cdot \beta_t -\t\EE \bigl[ \partial_\mu \sigma(t,\theta_t)(\t X_t)\cdot \t V_t \bigr] \biggr]dW_t\\
&=V^{\epsilon,1}_tdt + V^{\epsilon,2}_t dW_t.
\end{split}
\end{equation} 
Now for each $t\in[0,T]$ and each $\epsilon>0$, we have:
\begin{equation*}
\begin{split}
\frac{1}{\epsilon}\left[b(t,\theta_t^{\epsilon})-b(t,\theta_t)\right]
&=\int_0^1\partial_x b\bigl(t,\theta_t^{\lambda,\epsilon}\bigr) \cdot (V^\epsilon_t+V_t)d\lambda
+\int_0^1\partial_\alpha b\bigl(t,\theta_t^{\lambda,\epsilon}\bigr) \cdot \beta_t d\lambda
\\
&\hspace{15pt}+\int_0^1\t\EE\bigl[
\partial_ \mu b \bigl(t,\theta_t^{\lambda,\epsilon}\bigr)\bigl(\t X_t^{\lambda,\epsilon}\bigr)\cdot (\t V^\epsilon_t+\t V_t) \bigr] d\lambda,
\end{split}
\end{equation*} 
where, in order to simplify a little bit the notation, we have set $X^{\lambda,\epsilon}_t=X_t+\lambda\epsilon(V^\epsilon_t+V_t)$, 
$\alpha^{\lambda,\epsilon}_t=\alpha_t+\lambda\epsilon\beta_t$
and $\theta_{t}^{\lambda,\epsilon} = (X_{t}^{\lambda,\epsilon},\PP_{X_{t}^{\lambda,\epsilon}},\alpha_{t}^{\lambda,\epsilon})$. Computing the `$dt$'-term, we get:
\begin{equation*}
\label{fo:VVepsilon}
\begin{split}
V^{\epsilon,1}_t&=\int_0^1\partial_xb\bigl(t,\theta_{t}^{\lambda,\epsilon}\bigr) \cdot V^\epsilon_t d\lambda
+\int_0^1\t\EE \bigl[ \partial_\mu b\bigl(t,\theta^{\lambda,\epsilon}_t\bigr)(\t X^{\lambda,\epsilon}_t)\cdot\t V^\epsilon_t \bigr] d\lambda
\\
&\hspace{15pt}
+\int_0^1 \bigl[\partial_xb\bigl(t,\theta_{t}^{\lambda,\epsilon}\bigr)-\partial_x b(t,\theta_{t}) \bigr] \cdot V_t d\lambda
+\int_0^1\bigl[\partial_\alpha b\bigl(t,\theta_{t}^{\lambda,\epsilon} \bigr)-\partial_\alpha b(t,\theta_t) \bigr] \cdot \beta_t d\lambda
\\
&\hspace{30pt}+\int_0^1\t\EE
\bigl[
\bigl( \partial_\mu b \bigl(t,\theta^{\lambda,\epsilon}_t\bigr)
(\t X^{\lambda,\epsilon}_t)
-\partial_{\mu} b(t,\theta_t)(\t X_t) \bigr)\cdot\t V_t \bigr]
 d\lambda
 \\
&= \int_0^1\partial_xb\bigl(t,\theta^{\lambda,\epsilon}\bigr) \cdot V^\epsilon_t d\lambda
+\int_0^1\t\EE \bigl[ \partial_\mu b\bigl(t,\theta_{t}^{\lambda,\epsilon} \bigr)(\t X^{\lambda,\epsilon}_t)\cdot\t V^\epsilon_t \bigr] d\lambda
 +I^{\epsilon,1}_t+ I^{\epsilon,2}_t+ I^{\epsilon,3}_t.
\end{split}
\end{equation*} 
The three last terms of the above right hand side are bounded in $L^2([0,T] \times \Omega)$, uniformly in $\epsilon$. Indeed, the Lipschitz regularity of $b$ implies that $\partial_{x} b$ and $\partial_{\alpha} b$ are bounded and that $\partial_{\mu} b(t,x,\PP_{X},\alpha)(X)$ is bounded in $L^2(\Omega;\RR^d)$, uniformly in $(t,x,\alpha) \in [0,T] \times \RR^d \times A$ and $X \in L^2(\Omega,\RR^d)$. Next, we treat the diffusion part $V^{\epsilon,2}_t$  in the same way using Jensen's inequality and Burkholder-Davis-Gundy's inequality to control the quadratic variation of the stochastic integrals.  Consequently, going back to \eqref{fo:Vepsilon}, we see that, for any $S \in [0,T]$,
\begin{equation*}
\EE\sup_{0\le t\le S}|V^\epsilon_t|^2 \leq c' + c' \int_0^S\EE\sup_{0\le s\le t}|V^\epsilon_s|^2 dt,
\end{equation*} 
where as usual $c' >0$ is a generic constant which can change from line to line. Applying Gronwall's inequality, we deduce that 
$\EE\sup_{0\le t\le T}|V^\epsilon_t|^2 \leq c'$. Therefore, we have  
\begin{equation*}
\lim_{\epsilon \searrow 0} \EE \bigl[ \sup_{0 \leq \lambda \leq 1}  \sup_{0 \le t \le T} \bigl\vert X_{t}^{\epsilon,\lambda}
- X_{t} \bigr\vert^2 \bigr] = 0. 
\end{equation*}
We then prove that $I^{1,\epsilon}$, $I^{3,\epsilon}$ and $I^{3,\epsilon}$ converge to $0$ in $L^2([0,T]\times\Omega)$ as $\epsilon\searrow 0$. Indeed, 
\begin{eqnarray*}
\EE\int_0^T|I^{\epsilon,1}_t|^2dt&=&\EE\int_0^T\left|\int_0^1[\partial_xb\bigl(t,\theta^{\lambda,\epsilon}_t \bigr)-\partial_x b(t,\theta_t)]V_t d\lambda
\right|^2dt
\\
&\le&\EE\int_0^T\int_0^1|\partial_xb \bigl(t,\theta_{t}^{\lambda,\epsilon} \bigr)-\partial_x b(t,\theta_{t})|^2|V_t|^2 d\lambda dt.
\end{eqnarray*} 
Since the function $\partial_{x} b$ is bounded and continuous in $x$, $\mu$ and $\alpha$, the above right-hand side converges to $0$ as $\epsilon\searrow 0$. A similar argument applies to $I^{\epsilon,2}_{t}$ and $I^{\epsilon,3}_{t}$. Again, we treat the diffusion part $V^{\epsilon,2}_t$  in the same way using Jensen's inequality and Burkholder-Davis-Gundy's inequality.  Consequently, going back to \eqref{fo:Vepsilon}, we finally see that, for any $S \in [0,T]$,
\begin{equation*}
\EE\sup_{0\le t\le S}|V^\epsilon_t|^2dt 
\le c' \int_0^S\EE\sup_{0\le s\le t}|V^\epsilon_s|^2 dt+\delta_\epsilon
\end{equation*} 
where $\lim_{\epsilon\searrow 0}\delta_\epsilon=0$.
Finally, we get the desired result applying Gronwall's inequality.
\end{proof}

\vskip 2pt\noindent
We now compute the G\^ateaux derivative of the objective function. 

\begin{lemma}
\label{le:gateaux}
The function $\alpha\mapsto J(\alpha)$ is G\^ateaux differentiable and its derivative in the direction $\beta$ is given by:
\begin{equation}
\label{fo:gateaux}
\begin{split}
&\frac{d}{d\epsilon}J(\alpha+\epsilon\beta)\big|_{\epsilon=0} =\EE\int_0^T \bigl[\partial_x f(t,\theta_t) \cdot V_t+\t\EE
[ \partial_\mu 
 f(t,\theta_t)(\t X_t)\cdot\t V_t ]+\partial_\alpha f(t,\theta_t) \cdot \beta_t \bigr]dt
 \\
&\phantom{?????????????????????????} +\EE \bigl[ \partial_x g(X_T,\PP_{X_T}) \cdot V_T + \t\EE [\partial_\mu g(X_T,\PP_{X_T})(\t X_T)\cdot \t V_T] \bigr].
\end{split}
\end{equation}
\end{lemma} 
\begin{proof}
We use freely the notation introduced in the proof of the previous lemma.
\begin{equation}
\label{eq:3:1:1}
\begin{split}
\frac{d}{d\epsilon}J(\alpha+\epsilon\beta)\big|_{\epsilon=0}&=\lim_{\epsilon\searrow 0}\frac{1}{\epsilon}\EE\int_0^T
\bigl[f \bigl(t,\theta_t^\epsilon\bigr)- f(t,\theta_t) \bigr]dt
\\
&\phantom{???????????????} + \lim_{\epsilon\searrow 0}\frac{1}{\epsilon}\EE \bigl[g(X_T^\epsilon,\PP_{X_T^\epsilon})-g(X_T,\PP_{X_T}) \bigr].
\end{split}
\end{equation}
Computing the two limits separately we get
\begin{equation*}
\begin{split}
&\lim_{\epsilon\searrow 0}\frac{1}{\epsilon}\EE\int_0^T \bigl[f \bigl(t,\theta_t^\epsilon \bigr)- f(t,\theta_t) \bigr]dt\\
&\hspace{15pt}=\lim_{\epsilon\searrow 0}\frac{1}{\epsilon}\EE\int_0^T\int_0^1\frac{d}{d\lambda}f(t,\theta^{\lambda,\epsilon}_t)d\lambda dt \\
&\hspace{15pt}=\lim_{\epsilon\searrow 0}\EE\int_0^T\int_0^1 \bigl[\partial_x f \bigl(t,\theta^{\lambda,\epsilon}_t \bigr) \cdot (V^\epsilon_t+V_t)\\
&\hspace{50pt}+\t\EE \bigl[\partial_\mu f \bigl(t,\theta_{t}^{\lambda,\epsilon} \bigr)(\t X_t^{\lambda,\epsilon})\cdot (\t V^\epsilon_t+\t V_t) \bigr]  +\partial_\alpha f \bigl(t,\theta_{t}^{\lambda,\epsilon} \bigr) \cdot \beta_t
\bigr]d\lambda dt \\
&\hspace{15pt}=\EE\int_0^T \bigl[\partial_x f(t,\theta_t) \cdot V_t+\t\EE \bigl[\partial_\mu f(t,\theta_t)(\t X_t)\cdot \t V_t \bigr]+\partial_\alpha f(t,\theta_t) \cdot \beta_t \bigr]dt
\end{split}
\end{equation*}
using the hypothesis on the continuity and growth of the derivatives of $f$, the uniform convergence proven in the previous lemma and standard uniform integrability arguments. The second term in \eqref{eq:3:1:1} is tackled in a similar way. 
\end{proof}

\vskip 2pt\noindent
Observing that the conditions (\ref{eq:L2:a}--\ref{eq:L2:b}) are satisfied under (A1--4), 
the duality relationship is given by:
\begin{lemma}
\label{le:duality}
Given $(Y_{t},Z_{t})_{0 \leq t \leq T}$
as in Definition \ref{de:adjoint}, it holds:
\begin{eqnarray}
\label{fo:duality}
\EE[Y_T \cdot V_T]&=&\EE\int_0^T \bigl[Y_t \cdot \bigl( \partial_\alpha b(t,\theta_t) \cdot \beta_t \bigr) 
+Z_t \cdot \bigl( \partial_\alpha\sigma(t,\theta_t) \cdot \beta_t \bigr) 
\nonumber\\
&&\phantom{????????}-\partial_xf(t,\theta_t) \cdot V_t -
\t\EE \bigl[ \partial_\mu f(t,\theta_t)(\t X_t)\cdot \t V_t \bigr]\bigr]\,dt.
\end{eqnarray}
\end{lemma} 
\begin{proof}
Letting $\Theta_{t} = (X_{t},\PP_{X_{t}},Y_{t},Z_{t},\alpha_{t})$ and 
using the definitions \eqref{fo:Voft} of the variation process $ V$, and \eqref{fo:adjoint} or \eqref{fo:adjoint'} of the adjoint process $ Y$, integration by parts gives:
\begin{equation*}
\begin{split}
Y_T \cdot V_T &=Y_0 \cdot V_0+\int_0^TY_t \cdot dV_t+\int_0^TdY_t \cdot V_t+\int_0^Td[ Y, V]_t
\\
&= M_{T} + \int_0^T \biggl[Y_t \cdot \bigl( \partial_x b(t,\theta_t) \cdot V_t \bigr) 
+Y_t \cdot  \t\EE \bigl[ \partial_\mu b(t,\theta_t)(\t X_t)\cdot\t V_t \bigr]
+Y_t \cdot \bigl( \partial_\alpha b(t,\theta_t) \cdot \beta_t \bigr) 
\\
&\hspace{45pt}-\partial_x H(t,\Theta_t) \cdot V_t-\t\EE \bigl[ \partial_\mu H(t,\t \Theta_t)(X_t)\cdot V_t \bigr]
\\
&\hspace{45pt}+Z_t \cdot \bigl( \partial_x\sigma(t,\theta_t) \cdot V_t \bigr)+Z_t \cdot \t\EE \bigl[ \partial_\mu \sigma(t,\theta_t)(\t X_t)\cdot\t V_t \bigr]+Z_t \cdot \bigl( \partial_\alpha\sigma(t,\theta_t) \cdot \beta_t \bigr) \biggr]dt,
\end{split}
\end{equation*}
where $(M_t)_{0\le t\le T}$ is a mean zero integrable martingale. By taking expectations on both sides and applying Fubini's theorem:
\begin{equation*}
\begin{split}
&\EE\t\EE \bigl[ \partial_\mu H(t,\t \Theta_t)(X_t)\cdot V_t \bigr]
\\
&\phantom{?}=\EE\t\EE \bigl[ \partial_\mu H(t,\Theta_t)(\t X_t)\cdot \t V_t \bigr]
\\
&\phantom{?}=\EE\t\EE \bigl[ \bigl( \partial_\mu b(t,\theta_t)(\t X_t)\cdot\t V_t \bigr) \cdot Y_t+ \bigl( \partial_\mu\sigma(t,\theta_t)(\t X_t)\cdot\t V_t \bigr) \cdot Z_t+\partial_\mu f(t,\theta_t)(\t X_t)\cdot\t V_t \bigr].
\end{split}
\end{equation*}
By commutativity of the inner product, cancellations occur and we get the desired equality \eqref{fo:duality}.
\end{proof}

By putting together the duality relation \eqref{fo:duality} and \eqref{fo:gateaux} we get:

\begin{corollary}
The G\^ateaux derivative of $J$ at $\alpha$ in the direction of $\beta$ can be written as:
\begin{equation}
\label{fo:Hgateaux}
\frac{d}{d\epsilon}J(\alpha+\epsilon\beta)\big|_{\epsilon=0}=\EE\left[\int_0^T \partial_\alpha H(t,X_t,\PP_{X_t},Y_t,\alpha_t) \cdot \beta_t \right]dt. 
\end{equation}
\end{corollary} 
\begin{proof}Using Fubini's theorem, the second expectation appearing in the expression \eqref{fo:gateaux} of the G\^ateaux derivative of $J$ given in Lemma \ref{le:gateaux}
 can be rewritten as:
\begin{eqnarray*}
&&\EE \bigl[ \partial_x g(X_T,\PP_{X_T}) \cdot V_T + \t\EE \bigl( \partial_\mu g(X_T,\PP_{X_T})(\t X_T)\cdot \t V_T \bigr) \bigr]\\
&&\phantom{??????????????}=\EE \bigl[ \partial_x g(X_T,\PP_{X_T}) \cdot V_T \bigr] + \EE\t\EE \bigl[\partial_\mu g(\t X_T,\PP_{\t X_T})(X_T)\cdot V_T \bigr]\\
&&\phantom{??????????????}=\EE[ Y_T\cdot V_T],
\end{eqnarray*}
and using the expression derived in Lemma \ref{le:duality} for $\EE[ Y_T\cdot V_T]$ in \eqref{fo:gateaux} gives the desired result.
\end{proof}
The main result of this subsection is the following theorem.

\begin{theorem}
Under the above assumptions, if we assume further that the Hamiltonian $H$ is convex in $\alpha$, the admissible control $\alpha\in\bA$ is optimal, $ X$ is the associated (optimally) controlled state, and $( Y,  Z)$ are the associated adjoint processes solving the adjoint equation \eqref{fo:adjoint}, then we have:
\begin{equation}
\label{fo:Pminimum}
\forall \alpha\in A,\quad H(t,X_t,Y_t,Z_t,\alpha_t)\le H(t,X_t,Y_t,Z_t,\alpha) \qquad a.e. \text{ in } t\in[0,T],\; \PP - a.s.
\end{equation}
\end{theorem}

\vskip 2pt\noindent\emph{Proof. }
Since $A$ is convex, given $\beta\in\bA$ we can choose the perturbation $\alpha^\epsilon_t=\alpha_t+\epsilon (\beta_t-\alpha_t)$
which is still in $\bA$ for $0\le \epsilon\le 1$. Since $\alpha$ is optimal, we have the inequality
$$
\frac{d}{d\epsilon}J(\alpha+\epsilon(\beta-\alpha))\big|_{\epsilon=0}=\EE\int_0^T\partial_\alpha H(t,X_t,Y_t,Z_t,\alpha_t)\cdot (\beta_t-\alpha_t)\ge 0.
$$
By convexity of the Hamiltonian with respect to the control variable $\alpha\in A$, we conclude that
$$
\EE\int_0^T[H(t,X_t,Y_t,Z_t,\beta_t)- H(t,X_t,Y_t,Z_t,\alpha_t)]dt\ge 0, 
$$
for all $\beta$. Now, if for a given (deterministic) $\alpha\in A$ we choose $\beta$ in the following way, 
$$
\beta_t(\omega)
= \begin{cases}
\alpha&\text{ if } (t,\omega)\in C\\
\alpha_t(\omega)&\text{ otherwise}
\end{cases}
$$
for an arbitrary progressively-measurable set $C\subset [0,T]\times\Omega$ (that is $C \cap [0,t] \in {\mathcal B}([0,t]) \otimes {\mathcal F}_{t}$ for any $t \in [0,T]$), we see that
$$
\EE\int_0^T{\bf 1}_C[H(t,X_t,Y_t,Z_t,\alpha)- H(t,X_t,Y_t,Z_t,\alpha_t)]dt\ge 0,
$$
from which we conclude that
$$
H(t,X_t,Y_t,Z_t,\alpha)- H(t,X_t,Y_t,Z_t,\alpha_t) \ge 0 \qquad dt\otimes d\PP \;a.e. \, ,
$$
which is the desired conclusion.
\qed

\subsection{\textbf{A Sufficient Condition}}
The necessary condition for optimality identified in the previous subsection can be turned into a sufficient condition for optimality under some technical assumptions.

\begin{theorem}
\label{th:pontryagin}
Under the same assumptions of regularity on the coefficients as before, let $\alpha\in\bA$ be an admissible control, $ X= X^{\alpha}$ the corresponding controlled state process, and $( Y, Z)$ the corresponding adjoint processes.
Let us also assume that for each $t\in [0,T]$
\begin{enumerate}
\item $\RR^d \times {\mathcal P}_{2}(\RR^d) \ni (x,\mu)\mapsto g(x,\mu)$ is convex;
\item $\RR^d \times {\mathcal P}_{2}(\RR^d) \times A \ni (x,\mu,\alpha)\mapsto H(t,x,\mu,Y_t,Z_t,\alpha)$ is convex $dt\otimes d\PP$ almost everywhere.
\end{enumerate} 
Moreover, if 
\begin{equation}
\label{fo:isaacs}
 H(t,X_t,\PP_{X_t},Y_t,Z_t, \alpha_t)=\inf_{\alpha\in A}H(t,X_t,\PP_{X_t},Y_t,Z_t,\alpha),\quad  dt\otimes d\PP \;\text{a.e.}
 \end{equation}
then $\alpha$ is an optimal control, i.e.
$J(\alpha)=\inf_{ \alpha'\in\bA}J( \alpha')
$. 
\end{theorem}

\vskip 2pt\noindent\emph{Proof. }
Let $\alpha'\in\bA$ be a generic admissible control, and $ X'= X^{\alpha'}$ the corresponding controlled state.
By definition of the objective function of the control problem we have:
\begin{equation}
\label{eq:24:1:1}
\begin{split}
J(\alpha)-J(\alpha')&= \EE \bigl[g(X_T,\PP_{X_T})-g(X'_T,\PP_{X'_T}) \bigr]+ \EE\int_0^T
\bigl[f(t,\theta_t)-f(t,\theta_t') \bigr]dt
\\
&= \EE \bigl[g(X_T,\PP_{X_T})-g(X'_T,\PP_{X'_T}) \bigr] +\EE\int_0^T \bigl[H(t,\Theta_t)-H(t,\Theta_t') \bigr]dt
\\
&\phantom{????????}- \EE\int_0^T \bigl\{ \bigl[b(t,\theta_t)-b(t,\theta_t') \bigr]\cdot Y_t + 
\bigl[\sigma(t,\theta_t)-\sigma(t,\theta'_t)]\cdot Z_t \bigr\} dt
\end{split}
\end{equation}
by definition of the Hamiltonian, where $\theta_{t} = (X_{t},\PP_{X_{t}},\alpha_{t})$ and 
$\Theta_{t} = (X_{t},\PP_{X_{t}},Y_{t},Z_{t},\alpha_{t})$ (and similarly for 
$\theta_{t}'$ and $\Theta_{t}'$). The function
$g$ being convex, we have 
$$
g(x,\mu)-g(x',\mu')\le (x-x')\cdot \partial_xg(x)+\t\EE \bigl[\partial_\mu g(x,\mu)(\t X)\cdot (\t X-\t {X'})\bigr],
$$ 
so that
\begin{equation}
\label{fo:g}
\begin{split}
&\EE[g(X_T,\PP_{X_T})-g(X'_T,\PP_{X'_T})]
\\
&\le \EE \bigl[\partial_xg(X_T,\PP_{X_T}) \cdot (X_T-X'_T) + \t\EE \bigl[\partial_\mu g(X_T,\PP_{X_{T}})(\t X_T)  
\cdot (\t X_T-\t {X'_T}) \bigr]
\bigr]
\\
&= \EE \bigl[ \bigl( \partial_xg(X_T,\PP_{X_T})+ \t\EE[\partial_\mu g(\t X_T, \PP_{X_{T}})(X_T) ] \bigr) \cdot (X_T- X'_T) \bigr]
\\
&= \EE \bigl[Y_{T} \cdot (X_T-X'_T) \bigr] = \EE \bigl[(X_T-X'_T)\cdot Y_T \bigr],
\end{split}
\end{equation}
where we used Fubini's theorem and the fact that the `tilde random variables' are independent copies of the `non-tilde variables'. Using the adjoint equation  and taking the expectation, we get:
\begin{equation*}
\begin{split}
&\EE \bigl[ (X_T-X'_T)\cdot Y_T \bigr] 
\\
&=\EE \biggl[ \int_0^T(X_t-X'_t) \cdot dY_t+\int_0^TY_t \cdot d[X_t- X'_t]
+\int_0^T[\sigma(t,\theta_t)-\sigma(t,\theta_t')]\cdot  Z_tdt \biggr]
\\
&=- \EE \int_0^T
\bigl[ 
\partial_xH(t,\Theta_{t}) \cdot (X_t-X'_t) + 
\t\EE \bigl[ \partial_\mu H(t,\t \Theta_t)(X_t) \bigr]  \cdot (X_t-X'_t)  \bigr] dt
\\
&\phantom{??????}+\EE \int_0^T 
\bigl[
[b(t,\theta_t)-b(t,\theta'_t)] \cdot Y_{t}   + [\sigma(t,\theta_t)-\sigma(t,\theta'_t)]\cdot  Z_t \bigr] dt,
\end{split}
\end{equation*}
where we used integration by parts and the fact that $Y_t$ solves the adjoint equation. 
Using Fubini's theorem and the fact that $\t \Theta_t$ is an independent copy of $\Theta_t$, the expectation of the second term in the second line  can be rewritten as
\begin{equation}
\label{eq:24:1:2}
\begin{split}
\EE \int_0^T \bigl\{ \t\EE \bigl[ \partial_\mu H(t,\t \Theta_t)(X_t) \bigr] \cdot (X_t-X'_t) \bigr\} dt
&= \EE \t \EE \int_0^T  \bigl\{ 
[ \partial_\mu H(t,\Theta_t)(\t X_t) ] \cdot (\tX_t-\tX'_t) \bigr\} dt
\\
&=\EE \int_0^T \t\EE \bigl[ \partial_\mu H(t,\Theta_t)(\t X_t)\cdot (\t X_t - \t X'_t) \bigr] dt.
\end{split}
\end{equation}
Consequently, by \eqref{eq:24:1:1}, \eqref{fo:g} and \eqref{eq:24:1:2}, we obtain
\begin{equation}
\label{eq:27:1:1}
\begin{split}
J(\alpha)-J(\alpha')&\le \EE\int_0^T[H(t,\Theta_t)-H(t,\Theta'_t)]dt
\\
&\phantom{????}-\EE\int_0^T \bigl\{ \partial_xH(t,\Theta_t) \cdot (X_t-X'_t) + 
\t\EE \bigl[ \partial_\mu H(t,\t \Theta_t)(X_t) \cdot (\t X_{t} - \t X'_{t} ) \bigr] \bigr\} dt
\\
&\le 0
\end{split}
\end{equation}
because of the convexity assumption on $H$, see in particular \eqref{fo:convexity}, and because of the criticality of the admissible control $(\alpha_t)_{0\le t\le T}$, see \eqref{fo:isaacs},
 which says the first order derivative in $\alpha$ vanishes.
\qed

\subsection{\textbf{Special Cases}}
\label{subs:examples}
We consider a set of particular cases which already appeared in the literature, and we provide the special forms of the Pontryagin maximum principle which apply in these
cases. We discuss only sufficient conditions for optimality for the sake of definiteness. The corresponding necessary conditions can easily be derived from the results of Subsection \ref{subsec:4:1}.

\subsubsection*{\textbf{Scalar Interactions}}
In this subsection we show how the model handled in \cite{AndersonDjehiche} appears as 
a specific example of our more general formulation.
We consider scalar interactions for which the dependence upon the probability measure of the coefficients of the dynamics of the state and the cost functions is through functions of scalar moments of the measure. More specifically, we assume that:
\begin{equation*}
\begin{array}{ll}
\displaystyle b(t,x,\mu,\alpha)= \h b(t,x,\langle\psi,\mu\rangle,\alpha) \quad &\displaystyle \sigma(t,x,\mu,\alpha)= \h\sigma(t,x,\langle\phi,\mu\rangle,\alpha)
\\
\displaystyle f(t,x,\mu,\alpha)= \h f(t,x,\langle\gamma,\mu\rangle,\alpha) \quad &\displaystyle g(x,\mu)= \h g(x,\langle\zeta,\mu\rangle)
\end{array}
\end{equation*}
for some scalar functions $\psi$, $\phi$, $\gamma$ and $\zeta$ with at most quadratic growth at $\infty$, and functions $\h b$, $\h \sigma$ and $\h f$ defined on $[0,T]\times \RR^d\times\RR\times A$ with values in $\RR^d$, $\RR^{d\times m}$ and $\RR$ respectively, and a real valued function $\h g$ defined on $\RR^d\times \RR$. We use the bracket notation $\langle h,\mu\rangle$ to denote the integral of the function $h$ with respect to the measure $\mu$. The functions $\h b$, $\h\sigma$, $\h f$ and $\h g$ are similar to the functions $b$, $\sigma$, $f$ and $g$ with the variable $\mu$, which was a measure, replaced by a numeric variable, say 
$r$.
Reserving the notation $H$ for the Hamiltonian we defined above, we have: 
$$
H(t,x,\mu,y,z,\alpha)=\h b(t,x,\langle \psi,\mu\rangle,\alpha)\cdot y +\h \sigma(t,x,\langle \phi,\mu\rangle,\alpha)\cdot z +\h f(t,x,\langle \gamma,\mu\rangle,\alpha).
$$
We then proceed to derive the particular form taken by the adjoint equation in the present situation. We start with the terminal condition as it is easier to identify.
According to \eqref{fo:adjoint}, 
it reads:
$$
Y_T=\partial_x g(X_T,\PP_{X_T}) + \t\EE[\partial_\mu g(\t X_T,\PP_{\t X_T})(X_T)].
$$
Since the terminal cost is of the form $g(x,\mu)=\h g(x,\langle\zeta,\mu\rangle)$, given our definition of differentiability with respect to the variable $\mu$, we know, as a generalization of \eqref{fo:Hofmu}, that $\partial_\mu g(x,\mu)(\,\cdot\,)$ reads
\begin{equation*}
\partial_\mu g(x,\mu)(x') = \partial_{r} \h g\bigl(x,\langle\zeta,\mu\rangle\bigr) \partial \zeta (x'), \quad x' \in \RR^d.
\end{equation*}
Therefore, the terminal condition $Y_{T}$ can be rewritten as
$$
Y_T=\partial_x \h g \bigl(X_T,\EE[\zeta(X_T)]\bigr) + \t\EE \bigl[\partial_{r} \h g \bigl(\t X_T,\EE[\zeta( X_T)] \bigr) \bigr] \partial \zeta(X_T)
$$
which is exactly the terminal condition used in \cite{AndersonDjehiche} once we remark that the  `tildes' can be removed since $\t X_T$ has the same distribution as $X_T$. Within this framework, convexity in $\mu$ is quite easy to check. Here is a typical example borrowed 
from \cite{AndersonDjehiche}: if $g$ and $\hat{g}$ do not depend on $x$, then the function ${\mathcal P}_{2}(\RR^d) \ni \mu \mapsto 
g(\mu) = \hat{g}(\langle \zeta,\mu \rangle)$ is convex if $\zeta$ is convex and $\hat{g}$ is non-decreasing and convex.

Similarly, $\partial_\mu H(t,x,\mu,y,z,\alpha)$ can be identified to the $\RR^d$-valued function defined by
\begin{eqnarray*}
&&\partial_\mu H(t,x,\mu,y,z,\alpha)(x')= \bigl[ \partial_{r} \h b(t,x,\langle \psi,\mu\rangle,\alpha)\odot y \bigr] \partial \psi(x') + 
\bigl[
\partial_{r} \h \sigma(t,x,\langle \phi,\mu\rangle,\alpha)\odot z\bigr] \partial \phi(x')
\\
&&\phantom{?????????????????????????????????????} +\partial_{r} \h f(t,x,\langle \gamma,\mu\rangle,\alpha)\;\partial \gamma(x')
\end{eqnarray*}
and the dynamic part of the adjoint equation \eqref{fo:adjoint} rewrites:
\begin{equation*}
\begin{split}
dY_t&=-\bigl\{ 
\partial_x \h b(t,X_t,\EE[\psi(X_t)],\alpha_t)\odot Y_t + \partial_x \h \sigma(t,X_t,\EE[\phi(X_t)],\alpha_t)\odot Z_t
\\
&\hspace{100pt} +
\partial_x \h f(t,X_t,\EE[\gamma(X_t)],\alpha_t) \bigr\} dt +Z_tdW_t
\\
&\hspace{15pt}- \bigl\{ 
\t \EE \bigl[
\partial_{r} \h b(t,\t X_t,\EE[\psi( X_t)],\t \alpha_t) \odot \t Y_{t}
\bigr]\partial \psi(X_t)+ \t\EE\bigl[ 
\partial_{r} \h \sigma(t,\t X_t,\EE[\phi( X_t)],\t \alpha_t)
\odot \t Z_{t}
\bigr]\partial \phi(X_t)
\\
&\hspace{100pt} +\t\EE \bigl[\partial_{r} \h f(t,\t X_t,\EE[\gamma( X_t)],\t \alpha_t) \bigr]
\partial \gamma(X_t) \bigr\} dt,
\end{split}
\end{equation*}
which again, is exactly the adjoint equation used in \cite{AndersonDjehiche} once we \emph{remove the `tildes'}.

\begin{remark}
The mean variance portfolio optimization example discussed in \cite{AndersonDjehiche} and the solution proposed in \cite{Bensoussanetal} and \cite{CarmonaDelarueLaChapelle} of the optimal control of linear-quadratic (LQ) McKean-Vlasov dynamics are based on the general form of the Pontryagin principle proven in this section as applied to the scalar interactions considered in this subsection. 
\end{remark}

\subsubsection*{\textbf{First Order Interactions}}
In the case of first order interactions, the dependence upon the probability measure is linear in the sense that the coefficients $b$, $\sigma$, $f$ and $g$ are given in the form 
\begin{equation*}
\begin{array}{ll}
\displaystyle b(t,x,\mu,\alpha)= \langle \h b(t,x,\,\cdot\,,\alpha),\mu\rangle \quad &\displaystyle \sigma(t,x,\mu,\alpha)= \langle\h\sigma(t,x,\,\cdot\,,\alpha),\mu\rangle
\\
\displaystyle f(t,x,\mu,\alpha)= \langle\h f(t,x,\,\cdot\,,\alpha),\mu\rangle \quad
&\displaystyle g(x,\mu)= \langle \h g(x,\,\cdot\,),\mu\rangle
\end{array}
\end{equation*}
for some functions  $\h b$, $\h \sigma$ and $\h f$ defined on $[0,T]\times \RR^d\times\RR^d\times A$ with values in $\RR^d$, $\RR^{d\times m}$ and $\RR$ respectively, and a real valued function $\h g$ defined on $\RR^d\times \RR^d$. The form of this dependence comes from the original derivation of the McKean-Vlasov equation as limit of the dynamics of a large system of particles evolving according to a system of stochastic differential equations with \emph{mean field} interactions  of the form
\begin{equation}
\label{fo:MFdynamics}
dX^{i}_t=\frac1N\sum_{j=1}^N\h b(t,X^{i}_t,X^{j}_t)dt+\frac1N\sum_{j=1}^N\h \sigma(t,X^{i}_t,X^{j}_t)dW_t^j, 
\quad i=1,\cdots, N, \quad 0\le t\le T,
\end{equation}
where $W^i$'s are $N$ independent standard Wiener processes in $\RR^d$. 
In the present situation the linearity in $\mu$ implies that  $\partial_\mu g(x,\mu)(x')=\partial_{x'}\h g(x,x')$ and similarly
\begin{eqnarray*}
&&\partial_\mu H(t,x,\mu,y,z,\alpha)(x')= \partial_{x'} \h b(t,x,x',\alpha)\odot y + \partial_{x'} \h \sigma(t,x,x',\alpha)\odot z
+\partial_{x'} \h f(t,x,x',\alpha),
\end{eqnarray*}
and the dynamic part of the adjoint equation \eqref{fo:adjoint} rewrites:
\begin{equation*}
\begin{split}
dY_t
&=-\t\EE \bigl[ \partial_x \h H(t,X_t,\t X_t,Y_t,Z_t,\alpha_t) + \partial_{x'} \h H(t,\t X_t, X_t,\t Y_t,\t Z_t,\t\alpha_t) \bigr]dt+Z_tdW_t,
\end{split}
\end{equation*}
if we use the obvious notation
$$
\h H(t,x,x',y,z,\alpha)=\h b(t,x,x',\alpha)\cdot y +\h \sigma(t,x,x',\alpha)\cdot z +\h f(t,x,x',\alpha),
$$
and the terminal condition is given by
$$
Y_T= \t\EE \bigl[\partial_x\h g(X_T,\t X_T)+\partial_{x'} \h g(\t X_T,X_T) \bigr].
$$

\section{Solvability of Forward-Backward Systems}
\label{se:fbsde}
We now turn to the application of the stochastic maximum principle to the solution of the optimal control of McKean-Vlasov dynamics. The strategy is to identify a minimizer of the Hamiltonian, and to use it in the forward dynamics and the adjoint equation. This creates a coupling between these equations, leading to the study of an  FBSDE of mean field type.  As explained in the introduction, the existence results proven in \cite{CarmonaDelarue_sicon} and \cite{CarmonaDelarue_ecp} do not cover some of the solvable models (such as the LQ models). Here we establish existence and uniqueness by taking advantage of the specific structure of the equation, inherited from the underlying optimization problem. Assuming that the terminal cost and the Hamiltonian satisfy the same convexity assumptions as in the statement of Theorem \ref{th:pontryagin}, we indeed prove that unique solvability holds by applying the \emph{continuation method}, originally exposed within the framework of FBSDEs in \cite{PengWu}. 
Some of the results of this section were announced in the note \cite{CarmonaDelarueLaChapelle}.

\subsection{Technical Assumptions}
We now formulate the assumptions we shall use from now on. These assumptions subsume the assumptions (A1-4) introduced in Sections \ref{se:mkv} and \ref{se:pontryagin}.
As it is most often the case in applications of Pontryagin's stochastic maximum principle, we choose $A= \RR^k$ and we consider a \emph{linear} model for the forward dynamics of the state.
\vskip 4pt

(B1) The drift $b$ and the volatility $\sigma$ are linear in $\mu$, $x$ and $\alpha$. They read
\begin{equation*}
\begin{split}
&b(t,x,\mu,\alpha) = b_{0}(t) + b_{1}(t)\o\mu + b_{2}(t) x + b_{3}(t) \alpha,
\\
&\sigma(t,x,\mu,\alpha) = \sigma_{0}(t) + \sigma_{1}(t) \o \mu + \sigma_{2}(t) x + \sigma_{3}(t) \alpha,
\end{split}
\end{equation*}
for some bounded measurable deterministic functions $b_{0}$, $b_{1}$, 
$b_{2}$ and $b_{3}$  with values in $\RR^d$, $\RR^{d \times d}$, $\RR^{d \times d}$ and $\RR^{d \times k}$
and $\sigma_{0}$, $\sigma_{1}$, $\sigma_{2}$ and $\sigma_{3}$ with values in $\RR^{d \times m}$, 
$\RR^{(d \times m) \times d}$, $\RR^{(d \times m) \times d}$ and $\RR^{(d \times m) \times k}$
(the parentheses around $d\times m$ indicating that $\sigma_{i}(t) u_{i}$ is seen as an element of $\RR^{d \times m}$ whenever $u_{i} \in \RR^d$, with $i=1,2$, or $u_{i} \in \RR^k$, with $i=3$), and where we use the notation 
$\o\mu=\int x\,d\mu(x)$ for the mean of a measure $\mu$.

\vskip 2pt

(B2) The functions $f$ and $g$ satisfy the same assumptions as in (A.3-4) in Section \ref{se:pontryagin} (with respect to some constant $L$). In particular, there exists a constant $\hat{L}$ such that
\begin{equation*}
\begin{split}
&\bigl\vert f(t,x',\mu',\alpha') - f(t,x,\mu,\alpha) \bigr\vert + \bigl\vert g(x',\mu') - g(x,\mu) \bigr\vert
\\
&\hspace{15pt} \leq \hat{L} \bigl[ 1 + \vert x' \vert + \vert x \vert + \vert \alpha' \vert +
\vert \alpha \vert + \|  \mu  \|_{2} + \|\mu' \|_{2} \bigr] \bigl[ \vert (x',\alpha') - (x,\alpha) \vert +
W_{2}(\mu',\mu) \bigr].
\end{split}
\end{equation*} 

(B3)  There exists a constant $\hat{c}>0$ such that the derivatives of $f$ and $g$ 
with respect to $(x,\alpha)$ and $x$ respectively
are $\hat{c}$-Lipschitz continuous with respect to $(x,\alpha,\mu)$ and $(x,\mu)$ respectively
(the Lipschitz property in the variable $\mu$ being understood in the sense of the 2-Wasserstein distance). Moreover, for any
$t \in [0,T]$, any  
$x,x' \in \RR^d$, any $\alpha,\alpha' \in \RR^k$, any 
$\mu,\mu' \in {\mathcal P}_{2}(\RR^d)$ and 
any $\RR^d$-valued random variables $X$ and $X'$ having $\mu$ and $\mu'$ as respective distributions, 
\begin{equation*}
\begin{split}
&\EE \bigl[ \vert \partial_{\mu} f(t,x',\mu',\alpha')(X') - \partial_{\mu} f(t,x,\mu,\alpha)(X) \vert^2 \bigr]
\leq \hat{c} \bigl( \vert (x',\alpha') - (x,\alpha) \vert^2 +
\EE \bigl[ \vert X'- X \vert^2\bigr] \bigr),
\\
 &\EE \bigl[ \vert \partial_{\mu} g(x',\mu')(X') - \partial_{\mu} g(x,\mu)(X) \vert^2 \bigr] 
 \leq \hat{c} \bigl( \vert x'-x \vert^2 + \EE \bigl[ \vert X'- X \vert^2 \bigr] \bigr). 
\end{split}
\end{equation*}
\vskip 2pt

(B4) The function $f$ is convex with respect to  $(x,\mu,\alpha)$ for $t$ fixed in such a way that, for some $\lambda >0$,
\begin{equation*}
\begin{split}
&f(t,x',\mu',\alpha') - f(t,x,\mu,\alpha) 
\\
&\hspace{15pt} - \partial_{(x,\alpha)} f(t,x,\mu,\alpha)\cdot (x'-x,\alpha'-\alpha)
- \tilde{\mathbb E} \bigl[\partial_{\mu} f(t,x,\mu,\alpha)(\t X) \cdot (\tilde{X}' - \tilde{X}) \bigr]
\\ 
 &\geq \lambda 
\vert \alpha' - \alpha \vert^2, 
\end{split}
\end{equation*} 
whenever $\tilde{X},\tilde{X}' \in L^2(\tilde{\Omega},\tilde{\mathcal A},\tilde{\mathbb P};\RR^d)$ with distributions $\mu$ and $\mu'$ respectively. The function $g$ is also assumed to be convex in $(x,\mu)$ (on the same model, but with $\lambda=0$). 
\vspace{4pt}

We refer to Subsection \ref{subse:joint} for a precise discussion about (B3) and (B4). 
By comparing \eqref{fo:Lip} with (B3), notice in particular that the liftings $L^2(\t \Omega; \RR^d) \ni \t X \mapsto f(t,x,\t \PP_{\t X}, \alpha)$ and $L^2(\t \Omega;\RR^d) \ni \t X \mapsto g(x,\t \PP_{\t X})$ have Lipschitz continuous derivatives. As a consequence, 
Lemma \ref{le:5} applies: For any $t \in [0,T]$, 
$x \in \RR^d$, $\mu \in {\mathcal P}_{2}(\RR^d)$ and $\alpha \in \RR^k$, there exist versions of
$\RR^d \ni x' \mapsto \partial_{\mu} f(t,x,\mu,\alpha)(x')$ and $\RR^d \ni x' \mapsto \partial_{\mu} g(x,\mu)(x')$ that are 
$\hat{c}$-Lipschitz continuous. 

Following Example \eqref{fo:Hofmu}, we also emphasize that $b$ and $\sigma$ obviously satisfy (B3). 

\subsection{The Hamiltonian and the Adjoint Equations}
The drift and the volatility being linear, the Hamiltonian has the form
\begin{equation*}
\begin{split}
H(t,x,\mu,y,z,\alpha) &= \bigl[b_0(t)+b_1(t)\bar{\mu}+b_2(t)x+b_3(t)\alpha\bigr]\cdot y  
\\
&\hspace{15pt}
+ \bigl[ \sigma_{0}(t) + \sigma_{1}(t) \bar{\mu} + \sigma_{2}(t) x + \sigma_{3}(t) \alpha \bigr] \cdot z
 + f(t,x,\mu,\alpha),
\end{split}
\end{equation*}
for $t \in [0,T]$, $x,y \in \RR^d$, $z \in \RR^{d \times m}$, $\mu \in \cP_2(\RR^d)$ and $\alpha \in \RR^k$.
Given $(t,x,\mu,y,z) \in [0,T] \times \RR^d \times {\mathcal P}_{2}(\RR^d) \times \RR^d \times \RR^{d \times m}$, the function $\RR^k \ni \alpha \mapsto H(t,x,\mu,y,z,\alpha)$ is strictly convex so that 
there exists a unique minimizer $\hat{\alpha}(t,x,\mu,y,z)$:
\begin{equation}
\label{fo:alphahat}
\hat{\alpha}(t,x,\mu,y,z) = \textrm{argmin}_{\alpha \in \RR^k} H(t,x,\mu,y,z,\alpha).
\end{equation}
Assumptions (B1-4) above being slightly stronger than the assumptions used in \cite{CarmonaDelarue_sicon}, we can follow the arguments given in the proof of Lemma 2.1 of \cite{CarmonaDelarue_sicon} in order to prove that,
for all $(t,x,\mu,y,z)\in [0,T] \times \RR^d \times {\mathcal P}_{2}(\RR^d) \times \RR^d \times \RR^{d \times m}$,
the function $[0,T] \times \RR^d \times {\mathcal P}_{2}(\RR^d) \times \RR^d \times \RR^{d \times m}
\ni (t,x,\mu,y,z)   \mapsto \hat{\alpha}(t,x,\mu,y,z)$ is measurable, locally bounded and Lipschitz-continuous with respect to 
$(x,\mu,y,z)$, uniformly in $t\in [0,T]$, the Lipschitz constant depending only upon $\lambda$, the supremum norms 
of $b_{3}$ and $\sigma_{3}$ and the Lipschitz 
constant of $\partial_{\alpha}f$ in $(x,\mu)$. Except maybe for the Lipschitz property with respect to the measure argument, these facts were explicitly proved in
\cite{CarmonaDelarue_sicon}. The regularity of $\hat{\alpha}$ with respect to $\mu$ follows from the following remark.
If $(t,x,y,z) \in  [0,T] \times \RR^d \times \RR^d \times \RR^{d \times m}$ is fixed and $\mu,\mu'$ are generic elements in ${\mathcal P}_{2}(\RR^d)$,
$\hat{\alpha}$ and $\hat{\alpha}'$ denoting the associated minimizers, we deduce from the convexity assumption (B4):
\begin{equation}
\label{eq:17:3:5}
\begin{split}
2\lambda \vert \hat{\alpha}' - \hat{\alpha} \vert^2
&\leq ( \hat{\alpha}' - \hat{\alpha})\cdot \bigl[ \partial_{\alpha} f \bigl(t,x,\mu,\hat{\alpha}'\bigr)
- \partial_{\alpha} f \bigl(t,x,\mu,\hat{\alpha} \bigr)\bigr]
\\
&= ( \hat{\alpha}' - \hat{\alpha}) \cdot \bigl[ \partial_{\alpha} H \bigl(t,x,\mu,y,z,\hat{\alpha}'\bigr)
- \partial_{\alpha} H \bigl(t,x,\mu,y,z,\hat{\alpha} \bigr)\bigr]
\\
&= ( \hat{\alpha}' - \hat{\alpha}) \cdot \bigl[ \partial_{\alpha} H \bigl(t,x,\mu,y,z,\hat{\alpha}'\bigr)
- \partial_{\alpha} H \bigl(t,x,\mu',y,z,\hat{\alpha}' \bigr) \bigr]
\\
&= (\hat{\alpha}' - \hat{\alpha}) \cdot \bigl[\partial_{\alpha} f \bigl(t,x,\mu,\hat{\alpha}'\bigr)
- \partial_{\alpha} f \bigl(t,x,\mu',\hat{\alpha}' \bigr)\bigr]
\\
&\leq C \vert \hat{\alpha}' - \hat{\alpha}\vert\;  W_{2}(\mu',\mu),
\end{split}
\end{equation}
the passage from the second to the third following from the identity 
\begin{equation*}
\partial_{\alpha} 
H(t,x,\mu,y,z,\hat{\alpha}) = \partial_{\alpha} H(t,x,\mu',y,z,\hat{\alpha}')=0.
\end{equation*}
For each admissible control $\alpha=(\alpha_t)_{0 \leq t \leq T}$, if we denote the corresponding solution of the state equation by $ X=(X_t^\alpha)_{0 \leq t \leq T}$, then the adjoint BSDE \eqref{fo:adjoint} introduced in Definition \ref{de:adjoint} reads:
\begin{equation}
\begin{split}
&dY_{t} = - \partial_{x} f\bigl(t,X_{t},\PP_{X_t},\alpha_t \bigr) dt 
- b_{2}^{\dagger}(t) Y_{t} dt - \sigma_{2}^{\dagger}(t) Z_{t} dt + Z_{t}dW_{t}
\\
&\hspace{30pt} - \tilde{\mathbb E} \bigl[\partial_{\mu} f\bigl(t,\tilde{X}_{t},\PP_{X_{t}},\t \alpha_t \bigr) (X_{t}) \bigr]dt - b^{\dagger}_{1}(t) {\mathbb E} [ Y_{t} ] - \sigma_{1}^{\dagger}(t)  \EE [  Z_{t} ] dt. 
\end{split}
\end{equation}
Given the necessary and sufficient conditions proven in the previous section, our goal is to use the control $(\hat\alpha_t)_{0 \leq t \leq T}$ defined by $\hat\alpha_t=\hat\alpha(t,X_t,\PP_{X_t},Y_t,Z_{t})$ where $\hat\alpha$ is the minimizer function constructed above and the process $(X_t,Y_t,Z_{t})_{0 \leq t \leq T}$ is a solution of the FBSDE
\begin{equation}
\label{fo:mkv_fbsde}
\begin{cases}
&\begin{split}
&dX_{t} = \bigl[b_0(t)+b_1(t)\EE[X_t]+b_2(t)X_t+b_3(t)\hat\alpha(t,X_t,\PP_{X_t},Y_t,Z_{t}) \bigr] dt 
\\
&\hspace{10pt}+ \bigl[ \sigma_{0}(t) + \sigma_{1}(t) \EE[X_{t}] + \sigma_{2}(t) X_{t} + \sigma_{3}(t)
\hat{\alpha}(t,X_{t},\PP_{X_{t}},Y_{t},Z_{t}) \bigr] dW_{t},
\end{split}
\\
&\begin{split}
&dY_{t} = - \bigl[ \partial_{x} f\bigl(t,X_{t},\PP_{X_t},\hat{\alpha}(t,X_{t},\PP_{X_{t}},Y_{t},Z_{t}) \bigr)  
+b_{2}^{\dagger}(t) Y_{t} 
+\sigma_{2}^{\dagger}(t) Z_{t} \bigr] dt
+ Z_{t}dW_{t}
\\
&\hspace{10pt} - \bigl\{ \tilde{\mathbb E} \bigl[\partial_{\mu} f\bigl(t,\tilde{X}_{t},\PP_{X_{t}},\hat{\alpha}(t,\tilde{X}_{t},\PP_{X_{t}},
\tilde{Y}_{t},\tilde{Z}_{t}) \bigr)(X_{t}) \bigr] + b_{1}^{\dagger}(t) \EE [ Y_{t} ] +\sigma_{1}^{\dagger}(t) \EE [ Z_{t} ] 
\bigr\} dt,
\end{split}
\end{cases}
\end{equation}
 with the initial condition $X_{0}=x_{0}$, for a given deterministic point $x_{0} \in \RR^d$, and the terminal condition $Y_{T} = \partial_{x} g(X_{T},\PP_{X_{T}}) + \tilde{\EE}[\partial_{\mu} g(\tilde{X}_{T},\PP_{X_t})(X_{T})]$. 

\subsection{Main Result} 
Here is the main existence and uniqueness result:
\begin{theorem}
\label{th:sol:mkv}
Under (B1-4), the forward-backward system \eqref{fo:mkv_fbsde} is uniquely solvable. 
\end{theorem}

\begin{proof}
The proof is an adaptation of the \emph{continuation method} used in \cite{PengWu} to handle standard FBSDEs satisfying appropriate monotonicity conditions. 
Generally speaking, it consists in proving that existence and uniqueness are kept preserved when the coefficients in \eqref{fo:mkv_fbsde}
are slightly perturbed. Starting from an initial case for which existence and uniqueness are known to hold, we then establish Theorem 
\ref{th:sol:mkv} by modifying iteratively the coefficients so that \eqref{fo:mkv_fbsde} is eventually shown to belong to the class of uniquely solvable systems. 

A natural and simple strategy then consists in modifying the coefficients in a linear way. Unfortunately, this might generate heavy notations.  For that reason, we use the following conventions. 

First, as in Subsection \ref{subsec:4:1}, the notation $(\Theta_{t})_{0 \leq t \leq T}$ stands for the generic notation for
denoting a process of the form $(X_{t},\PP_{X_{t}},Y_{t},Z_{t},\alpha_{t})_{0 \leq t \leq T}$ with values in $\RR^d \times {\mathcal P}_{2}(\RR^d)  \times \RR^d \times \RR^{d\times m} \times \RR^k$. 
We will then denote by ${\mathbb S}$ the space of processes 
$(\Theta_{t})_{0 \leq t \leq T}$ such that $(X_{t},Y_{t},Z_{t},\alpha_{t})_{0 \leq t \leq T}$ is $({\mathcal F}_{t})_{0 \leq t \leq T}$ progressively-measurable, $(X_{t})_{0 \leq t \leq T}$ and $(Y_{t})_{0 \leq t \leq T}$ have continuous trajectories, and 
\begin{equation}
\label{eq:norm S}
\| \Theta \|_{\mathbb S} =
\EE \biggl[ \sup_{0 \leq t \leq T} \bigl[ 
\vert X_{t} \vert^2 + \vert Y_{t} \vert^2 \bigr] + \int_{0}^T \bigl[ \vert Z_{t} \vert^2 + \vert \alpha_{t}\vert^2 \bigr] dt
\biggr]^{1/2} < +\infty. 
\end{equation}
Similarly, the notation 
$(\theta_{t})_{0 \leq t \leq T}$ is the generic notation for
denoting a process $(X_{t},\PP_{X_{t}},\alpha_{t})_{0 \leq t \leq T}$ with values in 
$\RR^d \times {\mathcal P}_{2}(\RR^d) \times \RR^k$. All the processes $(\theta_{t})_{0 \leq t \leq T}$ that are considered below appear as the restrictions of an \emph{extended} process 
$(\Theta_{t})_{0 \leq t \leq T} \in {\mathbb S}$. 

Moreover, we call an initial condition for \eqref{fo:mkv_fbsde} a square-integrable ${\mathcal F}_{0}$-measurable random variable $\xi$ with values in $\RR^d$, that is an element of $L^2(\Omega,{\mathcal F}_{0},\PP;\RR^d)$. Recall indeed that ${\mathcal F}_{0}$ can be chosen as a $\sigma$-algebra independent of $(W_{t})_{0 \leq t \leq T}$. 
In comparison with the statement of Theorem \ref{th:sol:mkv}, this permits to generalize the case when $\xi$ is deterministic.

Finally, we call an input for \eqref{fo:mkv_fbsde} a four-tuple ${\mathcal I}=(
({\mathcal I}^b_{t},
\Is_{t},\If_{t})_{0 \leq t \leq T},{\mathcal I}^g_{T})$,  $(\Ib)_{0 \leq t \leq T}$, 
$(\Is)_{0 \leq t \leq T}$ and $(\If)_{0 \leq t \leq T}$ being three square integrable progressively-measurable processes with values in $\RR^d$, $\RR^{d \times m}$ and $\RR^d$ respectively, and 
$\Ig$ denoting a square-integrable ${\mathcal F}_{T}$-measurable random variable with values in $\RR^d$. Such an input is specifically designed to be injected into the dynamics of \eqref{fo:mkv_fbsde}, $\Ib$ being plugged into the drift of the forward equation, $\Is$ into the volatility of the forward equation, $\If$ into the bounded variation term of the backward equation and $\Ig$ into the terminal condition of the backward equation. The space of inputs is denoted by ${\mathbb I}$. It is endowed with the norm:
\begin{equation}
\label{eq:norm I}
\| {\mathcal I} \|_{\mathbb I} =
\EE \biggl[ \vert \Ig_{T} \vert^2 + \int_{0}^T \bigl[ \vert \Ib_{t} \vert^2 + \vert \Is_{t} \vert^2 + \vert \If_{t} \vert^2 \bigr] dt
\biggr]^{1/2}. 
\end{equation}

\vspace{4pt}

We then put:

\begin{definition}
\label{le:small NL}
For any $\gamma\in [0,1]$, any  
$\xi \in L^2(\Omega,{\mathcal F}_{0},\PP;\RR^d)$ and any input ${\mathcal I} \in {\mathbb I}$, the FBSDE 
\begin{equation}
\label{eq:27:1:4}
\begin{split}
&dX_{t} = \bigl( \gamma b(t,\theta_{t}) + \Ib_{t} \bigr) dt + \bigl( \gamma \sigma(t,\theta_{t}) + \Is_{t} \bigr) dW_{t},
\\
&dY_{t} = - \bigl( \gamma \bigl\{ \partial_{x} H(t,\Theta_{t}) +
 \t \EE \bigl[ \partial_{\mu} H(t,\tilde{ \Theta}_{t})(X_{t}) \bigr] \bigr\}  + \If_{t}\bigr) dt
 + Z_{t} dW_{t}, \quad t \in [0,T],
\end{split}
\end{equation}
with the optimality condition
\begin{equation}
\label{eq:1:2:1}
\alpha_{t} = \hat{\alpha}(t,X_{t},\PP_{X_{t}},Y_{t},Z_{t}), \quad t \in [0,T],
\end{equation}
and  with $X_{0}=\xi$ as initial condition and 
$$Y_{T} = \gamma \bigl\{\partial_{x} g(X_{T},\PP_{X_{T}})
+ \t \EE [\partial_{\mu} g(\t X_{T},\PP_{X_{T}})(X_{T})] \bigr\} + \Ig_{T}$$ 
as terminal condition, 
is referred to as ${\mathcal E}(\gamma,\xi,{\mathcal I})$.

Whenever $(X_{t},Y_{t},Z_{t})_{0 \leq t \leq T}$ is a solution, the full process $(X_{t},\PP_{X_{t}},Y_{t},Z_{t},\alpha_{t})_{0 \leq t \leq T}$ is referred to as the  \emph{associated extended solution}.
\end{definition}

\begin{remark}
\label{rem:notation}
The way the coupling is summarized between the forward and backward equations in \eqref{eq:27:1:4} is a bit different from the way Equation \eqref{fo:mkv_fbsde} is written. In the formulation used in the statement of Lemma \ref{le:small NL}, the coupling between the forward and the backward equations follows from the optimality condition \eqref{eq:1:2:1}. Because of that optimality condition, 
the two formulations are equivalent: When $\gamma=1$ and ${\mathcal I}\equiv0$, the pair \emph{(\ref{eq:27:1:4}--\ref{eq:1:2:1})} coincides with \eqref{fo:mkv_fbsde}. 
\end{remark}

The following lemma is proved in the next subsection:

\begin{lemma}
\label{le:fixed point}
Given $\gamma \in  [0,1]$, we say the property $({\mathcal S}_{\gamma})$ holds true if, 
 for any $\xi \in L^2(\Omega,{\mathcal F}_{0},\PP;\RR^d)$ and any ${\mathcal I} \in {\mathbb I}$, 
the FBSDE ${\mathcal E}(\gamma,\xi,{\mathcal I})$ 
has a unique extended solution in ${\mathbb S}$. With this definition, 
 there exists $\delta_{0}>0$ such that, if $({\mathcal S}_{\gamma})$ holds true for some $\gamma \in [0,1)$, then $({\mathcal S}_{\gamma+\eta})$ holds true for any 
 $\eta \in (0,\delta_{0}]$ satisfying $\gamma + \eta \leq 1$. 
\end{lemma}

Given Lemma \ref{le:fixed point}, Theorem \ref{th:sol:mkv} follows from a straightforward induction as $({\mathcal S}_{0})$ obviously holds true. 
\end{proof}

\subsection{Proof of Lemma\ref{le:fixed point}}
%
The proof follows from Picard's contraction theorem. As in the statement, consider indeed $\gamma$ such that $({\mathcal S}_{\gamma})$ holds true. For $\eta >0$, $\xi \in L^2(\Omega,{\mathcal F}_{0},\PP;\RR^d)$ and ${\mathcal I} \in {\mathbb I}$, we then define a mapping $\Phi$ from ${\mathbb S}$ into itself whose fixed points coincide with 
the solutions of ${\mathcal E}(\gamma+\eta,\xi,{\mathcal I})$. 

The definition of $\Phi$ is as follows. Given a process $\Theta \in {\mathbb S}$, we denote by $\Theta'$ the extended solution
of the FBSDE ${\mathcal E}(\gamma,\xi,{\mathcal I}^{\prime})$ with 
\begin{equation*}
\begin{split}
&\Ibp_{t} = \eta b(t,\theta_{t}) + \Ib_{t},
\\
&\Isp_{t} = \eta \sigma(t,\theta_{t}) + \Ib_{t},
\\
&\Ifp_{t} = \eta \partial_{x} H(t,\Theta_{t}) + \eta \t \EE \bigl[ \partial_{\mu} H(t,\tilde{ \Theta}_{t})(X_{t}) \bigr] + \If_{t},
\\
&\Igp_{T} = \eta \partial_{x} g(X_{T},\PP_{X_{T}}) + \eta \t \EE \bigl[ 
\partial_{\mu} g(\t X_{T},\PP_{X_{T}})(X_{T}) \bigr]
+ \Ig_{T}. 
\end{split}
\end{equation*}
By assumption, it is uniquely defined and it belongs to ${\mathbb S}$, so that the mapping $\Phi : \Theta \mapsto \Theta'$ maps ${\mathbb S}$ into itself. It is then clear that a process $\Theta \in {\mathbb S}$ is a fixed point of $\Phi$ if and only if $\Theta$ is an extended solution of ${\mathcal E}(\gamma + \eta,\xi,{\mathcal I})$. The point is thus to prove that $\Phi$ is a contraction when $\eta$ is small enough. This is a consequence of the following lemma: 

\begin{lemma}
\label{le:stability}
Let $\gamma \in [0,1]$ such that $({\mathcal S}_{\gamma})$ holds true. Then, there exists a constant $C$, independent of 
$\gamma$, such that, for any $\xi,\xi' \in L^2(\Omega,{\mathcal F}_{0},\PP;\RR^d)$ and ${\mathcal I},{\mathcal I}' \in {\mathbb I}$, the respective extended solutions $\Theta$ and $\Theta'$ of ${\mathcal E}(\gamma,\xi,{\mathcal I})$ and ${\mathcal E}(\gamma,\xi',{\mathcal I}')$ satisfy:
\begin{equation*}
\| \Theta - \Theta' \|_{\mathbb S} \leq C \bigl( {\mathbb E} \bigl[ \vert \xi - \xi' \vert^2 \bigr]^{1/2}
+ \| {\mathcal I} - {\mathcal I}' \|_{\mathbb I} \bigr).  
\end{equation*}
\end{lemma}

Given Lemma \ref{le:stability}, we indeed check that $\Phi$ is a contraction when $\eta$ is small enough. Given $\Theta^1$ and 
$\Theta^2$ two processes in ${\mathbb S}$ and denoting by $\Theta^{\prime,1}$ and $\Theta^{\prime,2}$ their respective 
images by $\Phi$, we deduce from Lemma \ref{le:stability} that 
\begin{equation*}
\| \Theta^{\prime,1} - \Theta^{\prime,2} \|_{\mathbb S} \leq C \eta
 \| \Theta^{1} - \Theta^{2} \|_{\mathbb S},  
\end{equation*}
which is enough to conclude. 
\subsection{Proof of Lemma \ref{le:stability}}

The strategy follows from a mere variation on the proof of the classical stochastic maximum principle. 
With the same notations as in the statement and with the classical convention for expanding $\Theta$ as 
$(X_{t},\PP_{X_{t}},Y_{t},Z_{t},\alpha_{t})_{0 \leq t \leq T}$ and for letting 
$(\theta_{t}=(X_{t},\PP_{X_{t}},\alpha_{t}))_{0 \leq t \leq T}$, we then compute
\begin{equation*}
\begin{split}
\EE \bigl[ (X_T'-X_T)\cdot Y_T \bigr] 
&= \EE \bigl[ (\xi'-\xi)\cdot Y_0 \bigr] 
\\
&\hspace{5pt}- \gamma \biggl\{  \EE \int_0^T
\bigl[ 
\partial_xH(t,\Theta_{t}) \cdot (X_t'-X_t)  + 
\t\EE \bigl[ \partial_\mu H(t,\t \Theta_t)(X_t) \bigr]  \cdot (X_t'-X_t)  \bigr] dt
\\
&\hspace{20pt}- \EE \int_0^T 
\bigl[
[b(t,\theta_t')-b(t,\theta_t)] \cdot Y_{t}   + [\sigma(t,\theta_t')-\sigma(t,\theta_t)]\cdot  Z_t \bigr] dt \biggr\}
\\
&\hspace{5pt} - \biggl\{ 
\EE \int_{0}^T \bigl[ (X_{t}'-X_{t}) \cdot \If_{t} + (\Ib_{t} - \Ibp_{t}) \cdot Y_{t} + (\Is_{t}-\Isp_{t}) \cdot Z_{t} \bigr] 
dt \biggr\}
\\
&= T_{0} - \gamma T_{1} - T_{2}.  
\end{split}
\end{equation*}
Moreover, following \eqref{fo:g},
\begin{equation*}
\begin{split}
\EE \bigl[(X_T'-X_T)\cdot Y_T \bigr]
&= \gamma \EE \bigl[ \bigl( \partial_xg(X_T,\PP_{X_T})+ \t\EE[\partial_\mu g(\t X_T, \PP_{X_T})(X_{T})] \bigr) \cdot (X_T'- X_T) \bigr]
\\
&\hspace{15pt} + \EE \bigl[(\Igp_T-\Ig_T)\cdot Y_T \bigr]
\\
&\leq \gamma \EE\bigl[g(X_T',\PP_{X_T'})-g(X_T,\PP_{X_T})\bigr]  + \EE \bigl[(\Igp_T-\Ig_T)\cdot Y_T \bigr].
\end{split}
\end{equation*}
Identifying the two expressions right above and then repeating the proof of Theorem \ref{th:pontryagin}, we obtain
\begin{equation}
\label{eq:1:2:3}
\gamma J(\alpha') - \gamma J(\alpha) \geq 
\gamma \lambda \EE \int_{0}^T \vert \alpha_{t} - \alpha_{t}' \vert^2 dt + 
T_{0} - T_{2} + \EE \bigl[(\Ig_T-\Igp_T)\cdot Y_T \bigr]. 
\end{equation}
Now, we can reverse the roles of $\alpha$ and $\alpha'$ in \eqref{eq:1:2:3}. Denoting by $T_{0}'$ and $T_{2}'$ the 
corresponding terms in the inequality and then making the sum of both inequalities, we deduce that:
\begin{equation*}
2 \gamma \lambda \EE \int_{0}^T \vert \alpha_{t} - \alpha_{t}' \vert^2 dt
 + T_{0} + T_{0}' - (T_{2} + T_{2}') + \EE \bigl[(\Ig_T-\Igp_T)\cdot (Y_T-Y_{T}') \bigr] \leq 0.  
\end{equation*}
The sum $T_{2}+T_{2}'$ reads
\begin{equation*}
\begin{split}
&T_{2} + T_{2}' 
\\
&\hspace{15pt}= 
 \EE \int_0^T \bigl[
- (\If_{t}-\Ifp_{t}
)\cdot (X_{t} - X_{t}')
+  
(\Ib_{t} - \Ibp_{t} )
\cdot (Y_{t}  - Y_{t}' ) 
 + (\Is_{t} - \Isp_{t} )
 \cdot  (Z_t-Z_{t}') \bigr] dt.
\end{split}
\end{equation*}
Similarly,
\begin{equation*}
T_{0} + T_{0}' = - \EE \bigl[ (\xi - \xi') \cdot (Y_{0} - Y_{0}') \bigr]. 
\end{equation*}
Therefore, using Young's inequality, there exists a constant $C$ (the value of which may increase from line to line), $C$ being independent of $\gamma$, such that, for any $\varepsilon >0$,
\begin{equation}
\label{eq:2:2:3}
 \gamma \EE \int_{0}^T \vert \alpha_{t} - \alpha_{t}' \vert^2 dt
 \leq 
 \varepsilon 
  \| \Theta - \Theta^{\prime} \|_{\mathbb S}^2 
+ \frac{C}{\varepsilon} \bigl( \EE \bigl[ \vert \xi - \xi' \vert^2 \bigr] +
 \| {\mathcal I} - {\mathcal I}' \|_{\mathbb I}^2 \bigr).  
\end{equation}
Now, we observe by standard estimates for BSDEs that there exists a constant $C$, independent of $\gamma$, such that  
\begin{equation}
\label{eq:2:2:1}
\begin{split}
&\EE \biggl[ \sup_{0 \leq t \leq T} \vert Y_{t} - Y_{t}' \vert^2 + \int_{0}^T \vert Z_{t} - Z_{t}' \vert^2 dt \biggr]
\\
&\hspace{15pt} \leq C \gamma \EE \biggl[
\sup_{0 \leq t \leq T} \vert X_{t} - X_{t}' \vert^2
+  \int_{0}^T   \vert \alpha_{t} - \alpha_{t}' \vert^2  dt \biggr] + 
C \| {\mathcal I} - {\mathcal I}' \|_{\mathbb I}^2. 
\end{split}
\end{equation}
Similarly, 
\begin{equation}
\label{eq:2:2:2}
\EE \bigl[ \sup_{0 \leq t \leq T} \vert X_{t} - X_{t}' \vert^2  \bigr]
\leq \EE \bigl[ \vert \xi - \xi' \vert^2 \bigr] + C\gamma  \EE \int_{0}^T \vert \alpha_{t} - \alpha_{t}' \vert^2 dt + C 
\|   {\mathcal I} - {\mathcal I}'\|_{\mathbb I}^2. 
\end{equation}
From \eqref{eq:2:2:1} and \eqref{eq:2:2:2} and then from \eqref{eq:2:2:3}, we deduce that 
\begin{equation}
\begin{split}
&\EE \biggl[ \sup_{0 \leq t \leq T} \vert X_{t} - X_{t}' \vert^2 + \sup_{0 \leq t \leq T} \vert Y_{t} - Y_{t}' \vert^2 + \int_{0}^T \vert Z_{t} - Z_{t}' \vert^2 dt \biggr]
\\
&\hspace{15pt} \leq C\gamma  \EE \int_{0}^T \vert \alpha_{t} - \alpha_{t}' \vert^2 dt + C 
\bigl( \EE \bigl[ \vert \xi - \xi' \vert^2 \bigr] +
\|   {\mathcal I} - {\mathcal I}'\|_{\mathbb I}^2 \bigr)
\\
&\hspace{15pt} \leq  C \varepsilon 
  \| \Theta - \Theta^{\prime} \|_{\mathbb S}^2 
+ \frac{C}{\varepsilon} \bigl( \EE \bigl[ \vert \xi - \xi' \vert^2 \bigr] +
 \| {\mathcal I} - {\mathcal I}' \|_{\mathbb I}^2 \bigr). 
\end{split}
\end{equation}
Using the Lispchitz property of $\hat{\alpha}(t,\cdot,\cdot,\cdot,\cdot)$ and then choosing $\varepsilon$ small enough, we complete the proof.

\subsection{Decoupling Field}
The notion of decoupling field, also referred to as `FBSDE value function', plays a main role in the machinery of forward-backward equations, since it permits to represent the value $Y_{t}$
of the backward process at time $t$ as a function of the value $X_{t}$ of the forward process at time $t$. When the coefficients of the forward-backward equation are random, the decoupling field is a random field. When the coefficients are deterministic, the decoupling field is a deterministic function, which solves some corresponding partial differential equation. Here is the structure of the decoupling field in the McKean-Vlasov framework:

\begin{lemma}
\label{le:value function}
For any $t \in [0,T]$ and any $\xi \in L^2(\Omega,{\mathcal F}_{t},\PP;\RR^d)$, there exists a unique solution, denoted by 
$(X_{s}^{t,\xi},Y_{s}^{t,\xi},Z_{s}^{t,\xi})_{t \leq s\leq T}$, of \eqref{fo:mkv_fbsde} when set on $[t,T]$ 
with $X_{t}^{t,\xi}=\xi$ as initial condition. 

In this framework, for any $\mu \in {\mathcal P}_{2}(\RR^d)$, there exists a measurable mapping $u(t,\cdot,\mu) : \RR^d \ni x \mapsto  u(t,x,\mu)$ such that
\begin{equation}
\label{eq:27:2:11}
\PP \bigl( Y_{t}^{t,\xi} = u(t,\xi,\PP_{\xi}) \bigr) =1. 
\end{equation}
Moreover, there exists a constant $C$, only depending on the parameters in (B1-4), such that, for any $t \in [0,T]$ and any 
$\xi^1,\xi^2 \in L^2(\Omega,{\mathcal F}_{t},\PP;\RR^d)$,
\begin{equation}
\label{eq:27:1:2}
\EE \bigl[ \vert u(t,\xi^1,\PP_{\xi^1}) - u(t,\xi^2,\PP_{\xi^2}) \vert^2 \bigr] \leq C \EE \bigl[ \vert \xi^1 - \xi^2 \vert^2 \bigr]. 
\end{equation}
\end{lemma}

The proof is given right below. For the moment, we 
notice that the additional variable $\PP_{\xi}$ is for free in the above writing since we could set
$v(t,\cdot) = u(t,\cdot,\PP_{\xi})$ and then have $Y_{t}^{t,\xi} = v(t,\xi)$.
The additional variable $\PP_{\xi}$ is specified to emphasize the non-Markovian nature of the equation over the state space $\RR^d$: starting from two different initial conditions, the decoupling fields might not be the same, since the law of the initial conditions might be different. Keep indeed in mind that, in the Markovian framework, the decoupling field is the same for all possible initial conditions, thus yielding the connection with partial differential equations. Here the Markov property holds, but over the enlarged space $\RR^d \times {\mathcal P}_{2}(\RR^d)$, thus justifying the use of the extra variable $\PP_{\xi}$. 
Nevertheless, we often forget to specify the dependence upon $\PP_{\xi}$ in the sequel of the paper. 

An important fact is that the representation formula \eqref{eq:27:2:11} can be extended to the whole path:
\begin{proposition}
\label{prop:value function}
Under (B1-4), for any $\xi \in L^2(\Omega,{\mathcal F}_{0},\PP;\RR^d)$, there exists a measurable mapping $v: [0,T] \times \RR^d \rightarrow \RR^d$ such that
\begin{equation*}
\PP \bigl( \forall t \in [0,T], \ Y_{t}^{0,\xi} = v(t,X_{t}^{0,\xi}) \bigr) = 1.
\end{equation*}
It satisfies $\sup_{0 \leq t \leq T} \vert v(t,0) \vert < + \infty$. 
Moreover, 
there exists a constant $C$ such that 
$v(t,\cdot)$ is $C$-Lipschitz continuous for any $t \in [0,T]$.
\end{proposition}

We start with
\begin{proof}[Proof of Lemma \ref{le:value function}]
Given $t \in [0,T)$ and $\xi \in L^2(\Omega,{\mathcal F}_{t},\PP;\RR^d)$, existence and uniqueness to \eqref{fo:mkv_fbsde} when set on $[t,T]$ with $\xi$ as initial condition is a direct consequence of Theorem \ref{th:sol:mkv} (or, more precisely, of the proof of it since we are handling a random initial condition). Using as underlying filtration the augmented filtration 
${\mathbb F}^t$ generated by $\xi$ and by $(W_{s}-W_{t})_{t \leq s \leq T}$, we deduce that $Y_{t}^{t,\xi}$ coincides a.s. with a $\sigma(\xi)$-measurable $\RR^d$-valued random variable. In particular, there exists a measurable function $u_{\xi}(t,\cdot) : \RR^d \rightarrow \RR^d$ such that $\PP (Y_{t}^{t,\xi} = u_{\xi}(t,\xi))=1$. 

We now claim that the law of 
$(\xi,Y_{t}^{t,\xi})$ only depends upon the law of $\xi$. This directly follows from the version of the Yamada-Watanabe theorem for FBSDEs, see \cite{Delarue02}: Since uniqueness holds pathwise, it also holds in law so that, given two initial conditions with the same law, the solutions also have the same laws.  
Therefore, given another $\RR^d$-valued random vector $\xi'$ with the same law as $\xi$, it holds $(\xi,u_{\xi}(t,\xi)) \sim (\xi',u_{\xi'}(t,\xi'))$. In particular, for any measurable function $v : \RR^d \rightarrow \RR^d$, the random variables $u_{\xi}(t,\xi) - v(\xi)$ and $u_{\xi'}(t,\xi') - v(\xi')$ have the same law. Choosing $v=u_{\xi}(t,\cdot)$, we deduce that $u_{\xi'}(t,\cdot)$ and $u_{\xi}(t,\cdot)$ are a.e. equal under the probability measure $\PP_{\xi}$. Put it differently, denoting by $\mu$ the law of $\xi$, there exists an element $u(t,\cdot,\mu) \in L^2(\RR^d,\mu)$ such that $u_{\xi}(t,\cdot)$ and $u_{\xi'}(t,\cdot)$ coincide $\mu$ a.e. with $u(t,\cdot,\mu)$. Identifying $u(t,\cdot,\mu)$ with one of its version, this proves that 
\begin{equation*}
\PP \bigl( Y_{t}^{t,\xi} = u(t,\xi,\mu) \bigr) = 1. 
\end{equation*}
When $t>0$, we notice that, for any $\mu \in {\mathcal P}_{2}(\RR^d)$, there exists an ${\mathcal F}_{t}$-measurable random variable $\xi$ such that $\mu = \PP_{\xi}$. In such a case, the procedure we just described permits to define
$u(t,\cdot,\mu)$ for any $\mu \in {\mathcal P}_{2}(\RR^d)$. The situation may be different when $t=0$ as ${\mathcal F}_{0}$ may reduce to events of measure zero or one. In such a case, ${\mathcal F}_{0}$ can be enlarged without any loss of generality in order to support $\RR^d$-valued random variables of any arbitrarily prescribed distribution. 

The Lipschitz property 
\eqref{eq:27:1:2}
of $u(0,\cdot,\cdot)$ is a direct consequence of 
Lemma \ref{le:stability} with $\gamma=1$. By a time shift, the same argument applies to $u(t,\cdot,\cdot)$. 
\end{proof}

We now turn to
\begin{proof}[Proof of Proposition \ref{prop:value function}]
For simplicity, we just denote $(X_{t}^{0,\xi},Y_{t}^{0,\xi},Z_{t}^{0,\xi})_{0 \leq t \leq T}$
by $(X_{t},Y_{t},Z_{t})_{0 \leq t \leq T}$.
The proof is then a combination of Lemmas
\ref{le:5}
and \ref{le:value function}.  
Indeed, given $t \in (0,T]$, Lemma 
\ref{le:value function} says that the family $(u(t,\cdot,\mu))_{\mu \in {\mathcal P}_{2}(\RR^d)}$
satisfies \eqref{eq:28:2:5} since any $\mu \in {\mathcal P}_{2}(\RR^d)$ can be seen as the law of some ${\mathcal F}_{t}$-measurable random vector $\zeta$. Therefore, for $\mu=\PP_{X_{t}}$, we can find a mapping $w(t,\cdot)$ that is $C$-Lipschitz continuous (for the same $C$ as in \eqref{eq:27:1:2}) and that 
 coincides with $u(t,\cdot,\PP_{X_{t}})$ a.e. under the probability measure $\PP_{X_{t}}$. It satisfies
 \begin{equation}
 \label{eq:5:2:1}
 \forall t \in [0,T], \quad \PP \bigl( Y_{t} = w(t,X_{t}) \bigr) = 1, 
 \end{equation}
 since $Y_{t}=Y_{t}^{t,X_{t}}$.  In particular,
 \begin{equation}
 \label{eq:29:1:1}
 \begin{split}
\sup_{0 \leq t \leq T} \vert w(t,0) \vert &\leq  
 \sup_{0 \leq t \leq T} {\mathbb E} \bigl[ \vert Y_{t} \vert \bigr] +  \sup_{0 \leq t \leq T} 
 {\mathbb E} \bigl[\vert w(t,X_{t}) - w(t,0) \vert \bigr]
 \\
 &\leq \sup_{0 \leq t \leq T} {\mathbb E} \bigl[ \vert Y_{t} \vert \bigr] + C \sup_{0 \leq t \leq T} 
 {\mathbb E} \bigl[\vert X_{t} \bigr] < + \infty.  
 \end{split}
 \end{equation}

For any integer $n \geq 1$, we then let
\begin{equation*}
v^n(t,x) =  {\mathbf 1}_{[0,T/2^n]}(t) w\bigl( \frac{T}{2^n},x \bigr) + 
\sum_{k=2}^{2^n} {\mathbf 1}_{((k-1)T/2^n,kT/2^n]}(t) w\bigl( \frac{kT}{2^n},x \bigr), \quad t \in [0,T], \ x \in \RR^d. 
\end{equation*}
Denoting by $v^{n,i}$ the $i$th coordinate of $v^n$ for any $i \in \{1,\dots,d\}$, we also let
\begin{equation*}
v^i(t,x) = \limsup_{n \rightarrow + \infty} v^{n,i}(t,x), \quad t \in [0,T], \ x \in \RR^d, 
\end{equation*}
and then $v(t,x)= (v^1(t,x),\dots,v^d(t,x))$. As each of the $v^n$ is a Borel measurable function on $[0,T] \times \RR^d$, so is $v$. Similarly, $v$ satisfies \eqref{eq:29:1:1} and, for any $t \in [0,T]$, $v(t,\cdot)$ is $C$-Lipschitz continuous. 

Finally, we notice that, for any $t \in {\mathbb D}_{n} = \{k T/2^n, \ k \in \{1,\dots,2^n\}\}$, with $n \in {\mathbb N} \setminus \{0\}$, and for $\ell \geq n$,
\begin{equation*}
v^{\ell}(t,\cdot) = w(t,\cdot) = v(t,\cdot), 
\end{equation*}
so that, $w(t,\cdot)=v(t,\cdot)$ for any $t \in {\mathbb D} = \cup_{n \geq 1} {\mathbb D}_{n}$. Therefore,  
${\mathbb D}$ being countable, we deduce from \eqref{eq:5:2:1} that
the event 
\begin{equation*}
A = \bigl\{\omega \in \Omega : \forall t \in {\mathbb D}, Y_{t}(\omega) = w\bigl(t,X_{t}(\omega)\bigr) = v\bigl(t,X_{t}(\omega) \bigr) \bigr\}
\end{equation*}
has measure $\PP(A)=1$. On the event $A$, we notice that, for any $t \in (0,T]$,
\begin{equation*}
Y_{t} = \lim_{n \rightarrow + \infty} Y_{t_{n}} = \lim_{n \rightarrow + \infty} v(t_{n},X_{t_{n}}),
\end{equation*}
where $(t_{n})_{n \geq 1}$ is the sequence of points in ${\mathbb D}$ such that, for any $n \geq 1$, 
$t_{n} \in {\mathbb D}_{n}$ and $t_{n}-T/2^n < t \leq t_{n}$.  
Since $v(t_{n},\cdot)$ is $C$-Lipschitz continuous, we deduce that, on $A$,
\begin{equation*}
Y_{t} = \lim_{n \rightarrow + \infty} v(t_{n},X_{t}),
\end{equation*}
which is to say that the sequence $(v(t_{n},X_{t}))_{n \geq 1}$ is convergent. Now we observe that $v(t_{n},X_{t})$ is also 
$v_{n}(t,X_{t})$.
Therefore, the limit must coincide with 
$v(t,X_{t})$. This proves that, on the event $A$, $Y_{t} = v(t,X_{t})$ for any $t \in (0,T]$. By the same argument, the same holds true at $t=0$ ($0$ being handled apart for questions of notation since the definition of $v_{n}$ at time $0$ is rather specific).  
\end{proof}

\section{\textbf{Propagation of Chaos and Approximate Equilibrium}}
\label{se:chaos}

In this section, we show how the solution of the optimal control of stochastic dynamics of the McKean-Vlasov type can be used to
handle $N$-player games when $N$ tends to $+\infty$. 

Throughout this section,  assumptions (B1-4) are in force. 
For each integer $N\ge 1$, we consider a stochastic system whose time evolution is given by a system of $N$ coupled stochastic differential equations of the form
\begin{equation}
\label{fo:stateN}
dU_{t}^i = b \bigl(t,U_{t}^i,\bar{\nu}_t^{N},\beta_t^i \bigr) dt + \sigma(t,U_{t}^i,\bar{\nu}_{t}^N,\beta_{t}^i) dW_{t}^i, \quad 1 \leq i \leq N; 
 \qquad 
\bar{\nu}^N_{t} = \frac{1}{N} \sum_{j=1}^N \delta_{U_{t}^j},
\end{equation}
with $t \in [0,T]$ and $U_{0}^i = x_{0}$, $1 \leq i \leq N$. Here $((\beta_{t}^i)_{0 \leq t \leq T})_{1 \leq i \leq N}$ are $N$ $\RR^k$-valued processes that are progressively measurable with respect to the filtration generated by $(W^1,\dots,W^N)$ and have finite $L^2$ norms over $[0,T] \times \Omega$:
\begin{equation*}
\forall i \in \{1, \dots, N\}, \quad \EE \int_{0}^T \vert \beta_{t}^i \vert^2 dt < + \infty, 
\end{equation*}
where, for convenience, we have fixed an infinite sequence  $((W_{t}^i)_{0 \leq t \leq T})_{i\ge 1}$ of independent $m$-dimensional Brownian motions. 
One should think of $U^i_t$ as the (private) state at time $t$ of agent or player $i\in\{1,\cdots,N\}$,  $\beta^i_t$ being the action taken at time $t$ by player $i$. For each $1 \leq i \leq N$, we denote by
\begin{equation}
\label{fo:costN}
J^{N,i}(\beta^1,\dots,\beta^N) =  {\mathbb E}
\biggl[ g\bigl(U_{T}^i,\bar{\nu}_{T}^N
\bigr) + \int_{0}^T f(t,U_{t}^i,\bar{\nu}_{t}^N,\beta_{t}^i) dt \biggr],
\end{equation}
the cost to the $i$th player. We then recover the same set-up as in the case of the mean field game models studied in \cite{CarmonaDelarue_sicon}. Anyhow, the rule we apply for minimizing the cost is a bit different. The point is indeed to minimize the cost over exchangeable strategies: when the strategy $\u\beta=(\beta^1, \cdots, \beta^N)$ is exchangeable, the costs to all the players are the same and thus read as a common cost $J^{N,i}(\u \beta)=J^N(\u \beta)$. From a practical point of view, restricting the minimization to exchangeable strategies  means that the players are intended to obey a common policy, which is not the case in the standard mean field game approach.  

In this framework, one of our goal is to
compute the limit
\begin{equation*}
\lim_{N \rightarrow + \infty} \inf_{\u \beta} J^N(\u \beta),
\end{equation*}
the infimum being taken over exchangeable strategies. Another one is to identify, for each integer $N$, a specific set of $\varepsilon$-optimal strategies and the corresponding state evolutions. 
\subsection{Limit of the Costs and Non-Markovian Approximate Equilibriums}
Recall that we denote by $J$ the optimal cost:
\begin{equation}
\label{eq:optimalcost}
J = \EE \biggl[ g(X_{T},\mu_{T}) + \int_{0}^T f\bigl(t,X_{t},\mu_{t},\hat{\alpha}(t,X_{t},\mu_{t},Y_{t},Z_{t})\bigr) dt \biggr],
\end{equation} 
where $(X_{t},Y_{t},Z_{t})_{0 \leq t \leq T}$ is the solution to \eqref{fo:mkv_fbsde} with $X_{0}=x_{0}$ as initial condition, $(\mu_{t})_{0 \leq t \leq T}$ denoting the flow of marginal probability measures $\mu_t=\PP_{X_{t}}$, for $0 \leq t \leq T$. 

For the purpose of comparison, we introduce $(\bar X^1,\cdots,\bar X^N)$, each $\bar{X}^i$ standing for the solution of the forward equation in \eqref{fo:mkv_fbsde} when driven by the Brownian motion $W^i$. Put it differently, $(\bar X^1,\cdots,\bar X^N)$ solves the system 
\eqref{fo:stateN} when the empirical distribution $\bar\nu^N_t$ is replaced by $\mu_t$ and $\beta^i_{t}$ is given by
$\beta_{t}^i = \bar{\alpha}_{t}^i$ with
\begin{equation*}
\bar{\alpha}_{t}^i = \hat{\alpha}(t,\bar{X}_{t}^i,\mu_{t},\bar{Y}_{t}^i,\bar{Z}_{t}^i),
\end{equation*}
the pair $(\bar{Y}^i,\bar{Z}^i)$ solving the backward equation in \eqref{fo:mkv_fbsde} when driven by $W^i$. Pay attention that the processes $((\bar{\Theta}_{t}^i = (\bar{X}_{t}^i,\mu_{t},\bar{ Y}_{t}^i, \bar{Z}_{t}^i, \bar{\alpha}_{t}^i))_{0 \leq t \leq T})_{1 \leq i \leq N}$ are independent. 
\vspace{4pt}

Here is the first result:
\begin{theorem}
\label{th:limit cost}
Under assumptions (B1-4), 
\begin{equation*}
\lim_{N \rightarrow + \infty} \inf_{\u \beta} J^N( \u \beta) = J, 
\end{equation*}
the infimum being taken over exchangeable (square integrable) strategies $\u \beta=(\beta^1,\cdots,\beta^N)$.  
Moreover, the non-Markovian control $\underline{\bar{\alpha}}=(\bar{\alpha}^1,\cdots,\bar{\alpha}^N)$ is an approximate 
optimal control in the sense that 
\begin{equation*}
\lim_{N \rightarrow + \infty} J^N( \underline{\bar{\alpha}}) = J. 
\end{equation*}
\end{theorem}

\begin{proof}
The proof consists in comparing $J^N(\u \beta)$ to $J$ for a given exchangeable strategy $\u {\beta}$. Once again, it relies on a variant of the stochastic maximum principle exposed in Section \ref{se:pontryagin}. With the above notation, we indeed obtain
\begin{equation*}
J^N(\u \beta)-J
=\EE \bigl[ g(U^i_{T},\bar{\nu}^N_{T}) - g(\bX_{T}^i,\mu_{T}) \bigr]
+ \EE \biggl[ \int_{0}^T \bigl( f(s,U_{s}^i,\bar{\nu}_{s}^N,\beta_{s}^i)  
- f(s,\bX_{s}^i,\mu_{s},\balpha_{s}^{i})  \bigr) ds \biggr],
\end{equation*}
the identity holding true for any $1 \leq i \leq N$. Therefore, we can write
\begin{equation}
\label{eq:28:2:30}
J^N(\u \beta)-J
= T_{1}^i + T_{2}^i,
\end{equation} 
with
\begin{equation*}
\begin{split}
T_{1}^i &= \EE \bigl[( U^i_{T}-\bX^i_{T}) \cdot \bar{Y}^i_{T}  \bigr]
+ \EE \biggl[ \int_{0}^T \bigl( f(s,U_{s}^i,\bar{\nu}_{s}^N,\beta_{s}^i)  - f(s,\bX_{s}^i,\mu_{s},\balpha_{s}^{i})  \bigr) ds \biggr],
\\
T_{2}^i &= \EE \bigl[ g(U^i_{T},\bar{\nu}_{T}^N) - g(\bX_{T}^i,\mu_{T}) \bigr] 
 - \EE \bigl[ (U_{T}^i - \bX_{T}^i)\cdot \partial_{x} g(\bX_{T}^i,\mu_{T}) \bigr]
\\
&\phantom{??????????}-  \EE \t\EE \bigl[ (\t U_{T}^{i} - \t{\bX}_{T}^{i})\cdot \partial_{\mu} g(\bX_{T}^i,\mu_{T})(\t{\bX}_{T}^{i}) \bigr]
\\
&= T_{2,1}^i - T_{2,2}^i - T_{2,3}^i,
\end{split}
\end{equation*}
where 
we used Fubini's theorem with the independent copies denoted with a tilde `\ $\tilde{\cdot}$\ '. 

\subsubsection*{Analysis of $T_{2}^i$} Using the diffusive effect of independence, we claim
\begin{equation*}
\begin{split}
T_{2,3}^i &= \EE  \t\EE \bigl[ (\t U_{T}^{i} - \t{\bX}_{T}^{i}) \cdot \partial_{\mu} g(\bX_{T}^i,\mu_{T})(\t{\bX}_{T}^{i}) \bigr]
\\
&= \frac{1}{N} \sum_{j=1}^N 
\t\EE \bigl[ (\t U_{T}^{i} - \t{\bX}_{T}^{i}) \cdot \partial_{\mu} g(\t{\bX}_{T}^{j},\mu_{T})(\t{\bX}_{T}^{i})\bigr]
\\
&\hspace{15pt} + 
{\mathcal O} \biggl( \t\EE \bigl[ \vert \t U_{T}^{i} - \t{\bX}_{T}^{i}\vert^2 \bigr]^{1/2}  \t\EE \biggl[ \biggl\vert \frac{1}{N} \sum_{j=1}^N 
\partial_{\mu} g(\t{\bX}_{T}^{j},\mu_{T})(\t{\bX}_{T}^{i})- 
\EE \bigl[ \partial_{\mu} g(\bar{X}_{T}^{i},\mu_{T})(\t{\bX}_{T}^{i}) \bigr]
\biggr\vert^2 \biggr]^{1/2} \biggr)
\\
&= \frac{1}{N} \sum_{j=1}^N 
\EE \bigl[ ( U_{T}^{i} - \bX_{T}^{i})\cdot \partial_{\mu} g\bigl(\bX_{T}^{j},\mu_{T})(\bX_{T}^{i})
  \bigr] +  {\mathbb E} \bigl[ \vert U_{T}^{i} - {\bX}_{T}^{i}
\vert^2 \bigr]^{1/2}  {\mathcal O}(N^{-1/2}),
\end{split}
\end{equation*}
where ${\mathcal O}(\cdot)$ stands for the Landau notation. 
Therefore, taking advantage of the exchangeability in order to handle the remainder, we obtain
\begin{equation*}
\frac{1}{N} \sum_{i=1}^N T_{2,3}^i
= \frac{1}{N^2} \sum_{j=1}^N \sum_{i=1}^N 
\EE \bigl[ ( U_{T}^{i} - \bX_{T}^{i}) \cdot \partial_{\mu} g\bigl(\bX_{T}^{j},\mu_{T})(\bX_{T}^{i})
\bigr] +  {\mathbb E} \bigl[ \vert U_{T}^{1} - \bX_{T}^{1}
\vert^2 \bigr]^{1/2}  {\mathcal O}(N^{-1/2}).
\end{equation*}
Introducing a random variable $\vartheta$ from $(\tilde{\Omega},\tilde{\mathcal F},\tilde{\PP})$ into $\RR$ with uniform distribution on $\{1,\dots,N\}$ as done in the proof of Proposition \ref{pr:diff_empir}, we can write
\begin{equation*}
\frac{1}{N} \sum_{i=1}^N T_{2,3}^i
= \frac{1}{N} \sum_{j=1}^N  
\EE \tilde{\EE} \bigl[ (U_{T}^{\vartheta} - \bX_{T}^{\vartheta})\cdot \partial_{\mu} g\bigl(\bX_{T}^{j},\mu_{T})(\bX_{T}^{\vartheta})
 \bigr] +  {\mathbb E} \bigl[ \vert U_{T}^{1} - \bX_{T}^{1}
\vert^2 \bigr]^{1/2}  {\mathcal O}(N^{-1/2}).
\end{equation*}
Finally, defining the flow of empirical measures
\begin{equation*}
\bar{\mu}^N_{t} = \frac{1}{N} \sum_{j=1}^N \delta_{\bar{X}_{t}^j}, \quad t \in [0,T],
\end{equation*}
and using (B3), Propositions \ref{pr:diff_empir}
and \ref{prop:14:3:1}  and also Remark \ref{re:convergence} to estimate the distance $W_2(\bar\mu^N_T,\mu_T)$, the above estimate gives:
\begin{equation*}
\frac{1}{N} \sum_{i=1}^N T_{2,3}^i
= \frac{1}{N} \sum_{j=1}^N  
\EE \t\EE\bigl[ (U_{T}^{\vartheta} - \bX_{T}^{\vartheta})\cdot \partial_{\mu} g\bigl(\bX_{T}^{j},\bar{\mu}_{T}^N)(\bX_{T}^{\vartheta})
 \bigr] +  {\mathbb E} \bigl[ \vert U_{T}^{1} - \bX_{T}^{1}
\vert^2 \bigr]^{1/2}  {\mathcal O}(\ell_{N}(d)),
\end{equation*}
where we used the notation $\ell_{N}(d)$ for any function of $N$ which could be used as an upper bound for:
\begin{equation}
\label{eq:3:2:1}
\EE \bigl[ W_{2}^2(\bar{\mu}^N_{T},\mu_{T}) \bigr]^{1/2} + \biggl( \int_{0}^T 
\EE \bigl[ W_{2}^2(\bar{\mu}^N_{t},\mu_{t}) \bigr] dt \biggr)^{1/2} = {\mathcal O} (\ell_{N}(d)). 
\end{equation}
By Remark \ref{re:convergence}, the left-hand side tends to $0$ as $N$ tends to $+ \infty$, since the function
$[0,T] \ni t \mapsto \EE[ W_{2}^2(\bar{\mu}^N_{t},\mu_{t})]$ can be bounded independently of $N$. Therefore, $(\ell_{N}(d))_{N \geq 1}$ is always chosen as a sequence that converges to $0$ as $N$ tends to $+ \infty$. 
When 
$\sup_{0 \leq t \leq T} \vert \bar{X}_{t}^1 \vert$ has finite moment of order $d+5$, Remark \ref{re:convergence} says that 
$\ell_{N}(d)$ can be chosen as $N^{-1/(d+4)}$. In any case, we will assume that $\ell_{N}(d) \geq N^{-1/2}$. Going back to \eqref{eq:28:2:30},
\begin{equation*}
\begin{split}
\frac{1}{N} \sum_{i=1}^N T_{2}^i &= \frac{1}{N} \sum_{i=1}^N  
\biggl\{ \EE \bigl[ g(U^i_{T},\bar{\nu}_{T}^N) - g(\bX_{T}^i,\bar{\mu}_{T}^N) \bigr] 
- \EE \bigl[ (U_{T}^i - \bX_{T}^i)\cdot \partial_{x} g(\bX_{T}^i,\bar{\mu}_{T}^N)  \bigr]
\\
&\hspace{15pt} - \EE \t \EE \bigl[ (U_{T}^{\vartheta} - \bX_{T}^{\vartheta})\cdot \partial_{\mu}g
\bigl(\bX_{T}^{i},\bar{\mu}_{T}^N)(\bX_{T}^{\vartheta}) \bigr]
\biggr\} +  \bigl( 1 +  {\mathbb E} \bigl[ \vert U_{T}^{1} - \bX_{T}^{1}
\vert^2 \bigr]^{1/2} \bigr) {\mathcal O}(\ell_{N}(d)),
\end{split}
\end{equation*}
where we used local Lipschitz property of $g$ and Remark \ref{re:convergence} to replace $\mu_{T}$ by $\bar{\mu}^N_{T}$.

Noting that a.s. under $\PP$, the law of $U^{\vartheta}_{T}$ (resp. $\bar{X}^{\vartheta}_{T}$) under $\t \PP$ is the empirical distribution 
$\bar{\nu}_{T}^N$ (resp. $\bar{\mu}_{T}^N$), we can apply the convexity property of $g$, see
\eqref{fo:convexity}, to get
\begin{equation}
\label{eq:14:3:3}
\frac{1}{N} \sum_{i=1}^N T_{2}^i \geq  - \bigl( 1 + {\mathbb E} \bigl[ \vert U_{T}^{1} - \bX_{T}^{1}
\vert^2 \bigr]^{1/2} \bigr) {\mathcal O}( \ell_{N}(d)).
\end{equation}

\subsubsection*{Analysis of $T_{1}^i$} Using It\^o's formula and Fubini's theorem, we obtain
\begin{equation}
\label{eq:14:3:5}
\begin{split}
T_{1}^i &=  \EE \biggl[ \int_{0}^T 
\bigl( H(s,U_{s}^i,\bar{\nu}_{s}^N,\bar{Y}^i_s,\bar{Z}_{s}^i,\beta_{s}^i)  
- H(s,\bX_{s}^i,\mu_{s},\bar{Y}^i_s,\bar{Z}_{s}^i,\balpha_{s}^{i})  \bigr) ds \biggr]
\\
&\hspace{15pt}
- \EE \biggl[ \int_{0}^T (U_{s}^i - \bX_{s}^i)\cdot \partial_{x} H(s,\bX_{s}^i,\mu_{s},\bar{Y}_{s}^i,
\bar{Z}_{s}^i,\balpha_{s}^{i})
ds \biggr]
\\
&\hspace{15pt}
- \EE \t\EE \biggl[ \int_{0}^T (\t U_{s}^{i} - \t{\bX}_{s}^{i})\cdot 
\partial_{\mu} H(s,\bX_{s}^i,\mu_{s},\bar{Y}_{s}^i,\bar{Z}_{s}^i,\balpha_{s}^{i})(\t{\bX}_{s}^{i})
ds \biggr]
\\
&= T_{1,1}^i - T_{1,2}^i - T_{1,3}^i. 
\end{split}
\end{equation}
Using the local Lipschitz property of the Hamiltonian and \eqref{eq:3:2:1} and recalling that the limit process
$(\bar{X}^i_{t},\mu_{t},\bar{Y}^i_{t},\bar{Z}^i_{t},\bar{\alpha}_{t}^i)_{0 \leq t \leq T}$
has finite ${\mathbb S}$-norm (see \eqref{eq:norm S}), we get:
\begin{equation*}
T_{1,1}^i =  \EE \biggl[ \int_{0}^T 
\bigl(  H(s,U_{s}^i,\bar{\nu}_{s}^N,\bar{Y}^i_s,\bar{Z}_{s}^i,\beta_{s}^i)  
- H(s,\bX_{s}^i,\bar{\mu}_{s}^N,\bar{Y}^i_s,\bar{Z}_{s}^i,\balpha_{s}^{i})  \bigr) ds \biggr] + {\mathcal O}(\ell_{N}(d)).
\end{equation*}
Similarly, by exchangeability:
\begin{equation*}
\begin{split}
T_{1,2}^i
&= 
\EE \biggl[ \int_{0}^T 
( U_{s}^i - \bX_{s}^i)\cdot \partial_{x} H(s,\bX_{s}^i,\bar{\mu}_{s}^N,\bar{Y}_{s}^i,\bar{Z}_{s}^i,\balpha_{s}^{i})
 ds \biggr]  
+  \biggl(\EE \int_{0}^T \vert  U_{s}^i - \bX_{s}^i \vert^2 ds \biggr
 )^{1/2} {\mathcal O}(\ell_{N}(d)).
\end{split}
\end{equation*}
Finally, using the diffusive effect of independence, we have
\begin{equation}
\label{eq:30:1:2}
\begin{split}
\frac{1}{N} \sum_{i=1}^N T_{1,3}^i
&= \frac{1}{N} \sum_{i=1}^N \EE \t\EE \biggl[ \int_{0}^T ( U_{s}^{i} - {\bX}_{s}^{i})\cdot 
\partial_{\mu} H(s,\t\bX_{s}^i,\mu_{s},\tilde{\bar{Y}}_{s}^i,\tilde{\bar{Z}}_{s}^i,\t \balpha_{s}^{i})({\bX}_{s}^{i}) ds \biggr]
\\
&= \frac{1}{N^2} \sum_{j=1}^N \sum_{i=1}^N \EE \biggl[ \int_{0}^T ( U_{s}^{i} -\bX_{s}^{i})\cdot 
\partial_{\mu} H(s,{\bX}_{s}^j,\mu_{s}, \bar{Y}_{s}^j,\bar{Z}_{s}^j,{\balpha}_{s}^{j})(\bX_{s}^{i}) ds \biggr]\\
&\hspace{30pt}
+ \biggl( \EE \int_{0}^T \vert  U_{s}^{1} - \bX_{s}^{1} \vert^2 ds \biggr
 )^{1/2} {\mathcal O}(N^{-1/2}).
\end{split}
\end{equation}
By (B3), Propositions \ref{pr:diff_empir} and \ref{prop:14:3:1}, we have
\begin{equation*}
\begin{split}
\frac{1}{N} \sum_{i=1}^N T_{1,3}^i
&= \frac{1}{N}  \sum_{i=1}^N \EE \t\EE \biggl[ \int_{0}^T (U_{s}^{\vartheta} -\bX_{s}^{\vartheta})\cdot 
\partial_{\mu} H(s,\bX_{s}^{i},\mu_{s},\bar{Y}_{s}^{i},\bar{Z}_{s}^i,\balpha_{s}^{i})(\bX_{s}^{\vartheta}) ds \biggr]
\\
&\hspace{30pt}
+ \biggl( \EE \int_{0}^T \vert  U_{s}^{1}  - \bX_s^1 \vert^2 ds \biggr)^{1/2} {\mathcal O}(N^{-1/2})
\\
&= \frac{1}{N} \sum_{i=1}^N \EE \t \EE \biggl[ \int_{0}^T (U_{s}^{\vartheta} -\bX_{s}^{\vartheta})\cdot 
\partial_{\mu} H(s,\bX_{s}^{i},\bar{\mu}^N_{s},\bar{Y}_{s}^{i},\bar{Z}_{s}^i,\balpha_{s}^{i})(\bX_{s}^{\vartheta}) ds \biggr]
\\
&\hspace{30pt}
+ \biggl( \EE \int_{0}^T \vert  U_{s}^{1} - \bX_{s}^1 \vert^2 ds \biggr)^{1/2} {\mathcal O}(\ell_{N}(d)).
 \end{split}
 \end{equation*} 
In order to complete the proof, we evaluate the missing term in the Taylor expansion of $T_{1}^i$ in \eqref{eq:14:3:5}, namely
\begin{equation*}
\frac{1}{N} \sum_{i=1}^N \EE \biggl[ \int_{0}^T 
( \beta^i_{s} - \balpha_{s}^{i})\cdot \partial_{\alpha} H(s,\bX_{s}^i,\bar{\mu}^N_{s},\bar{Y}_{s}^i,\bar{Z}_{s}^i,\balpha_{s}^{i})
ds \biggr],
\end{equation*}
in order to benefit from the convexity of $H$. We use 
Remark \ref{re:convergence} once more:
\begin{equation}
\label{eq:30:1:1}
\begin{split}
&\EE \biggl[ \int_{0}^T 
(\beta^i_{s} - \balpha_{s}^{i}) \cdot \partial_{\alpha} H(s,\bX_{s}^i,\bar{\mu}_{s}^N,\bar{Y}_{s}^i,\bar{Z}_{s}^i,
\balpha_{s}^{i}) ds \biggr]
\\
&=  \EE \biggl[ \int_{0}^T 
(\beta^i_{s} - \balpha_{s}^{i}) \cdot \partial_{\alpha} H(s,\bX_{s}^i,\mu_{s},\bar{Y}_{s}^i,\bar{Z}_{s}^i,\balpha_{s}^{i}) ds \biggr] 
+  \biggl(  \EE \int_{0}^T \vert  \beta_{s}^i - \balpha_{s}^{i}  \vert^2 ds \biggr)^{1/2} {\mathcal O}(\ell_{N}(d))
\\
&=   \biggl(  \EE \int_{0}^T \vert  \beta_{s}^i - \balpha_{s}^{i}  \vert^2 ds \biggr)^{1/2} {\mathcal O}(\ell_{N}(d)),
\end{split}
\end{equation}
since $\bar{\alpha}$ is an optimizer for $H$.
Using the convexity of $H$ and taking advantage of the exchangeability, we finally deduce from 
\eqref{eq:14:3:5},
\eqref{eq:30:1:2} and \eqref{eq:30:1:1} that there exists a constant $c>0$ such that
\begin{equation*}
\begin{split}
\frac{1}{N} \sum_{i=1}^N T_{1}^i &\geq 
 c {\mathbb E} \int_{0}^T \vert  \beta_{s}^1 - \balpha_{s}^1 \vert^2 ds 
 \\
&\hspace{15pt} - 
{\mathcal O}(\ell_{N}(d)) \biggl( 1+ \sup_{0 \leq t \leq T} \EE  \bigl[ \vert  U_{t}^1 - \bX_{t}^1 \vert^2 \bigr] 
+ \EE \int_{0}^T  \vert  \beta_{s}^1 - \balpha_{s}^1   \vert^2  ds \biggr)^{1/2}.
\end{split}
\end{equation*}
By \eqref{eq:14:3:3} and \eqref{eq:28:2:30}, we deduce that 
\begin{equation*}
\begin{split}
J^N( \u \beta) &\geq J+ c {\mathbb E} \int_{0}^T \vert  \beta_{s}^1 - \balpha_{s}^1 \vert^2 ds 
\\
&\hspace{15pt}
- {\mathcal O}(\ell_{N}(d)) \biggl(  1+\sup_{0 \leq t \leq T} \EE  \bigl[ \vert  U_{t}^1 -\bX_{t}^1 \vert^2 \bigr] 
+ \EE \int_{0}^T  \vert  \beta_{s}^1 - \balpha_{s}^1  \vert^2  ds \biggr)^{1/2}.
\end{split}
\end{equation*}
 From the inequality
 \begin{equation*}
  \sup_{0 \leq t \leq T} \EE \bigl[\vert  U_{t}^1 -\bX_{t}^1 \vert^2 \bigr]  \leq C  \EE \int_{0}^T 
  \vert  \beta_{s}^1 - \balpha_{s}^1  \vert^2 ds, 
 \end{equation*}
which holds for some constant $C$ independent of $N$,
we deduce that  
\begin{equation}
\label{eq:3:2:2}
J^N(\u \beta) \geq J - C  \ell_{N}(d),
 \end{equation}
 for a possibly new value of $C$, which proves that 
 \begin{equation*}
 \liminf_{N \rightarrow + \infty} \inf_{\u \beta} J^N (\u \beta) \geq J. 
 \end{equation*}
 
In order to prove Theorem \ref{th:limit cost}, it thus remains to find a sequence of controls $(\u \beta^N)_{N \geq 1}$ such that 
 \begin{equation*}
 \limsup_{N \rightarrow + \infty} J^N (\u \beta^N) \leq J. 
 \end{equation*}
 Precisely, we show right below that 
 \begin{equation*}
 \limsup_{N \rightarrow + \infty} J^N (\underline{\bar{\alpha}}) \leq J, 
 \end{equation*}
 thus proving that $\underline{\bar{\alpha}} = (\bar{\alpha}^1,\dots,\bar{\alpha}^N)$ is an approximate equilibrium, but of non-Markovian type. Denoting by $(X^1,\dots,X^N)$ the solution of \eqref{fo:stateN} with $\beta^i_{t} = \bar{\alpha}^i_{t}$, classical estimates from the theory of propagation of chaos imply (see e.g. \cite{Sznitman} or \cite{JourdainMeleardWoyczynski}) that 
\begin{equation*}
 \sup_{0 \leq t \leq T} \EE \bigl[ \vert X^i_{t} - \bar{X}_{t}^i \vert^2 \bigr] =  \sup_{0 \leq t \leq T} \EE \bigl[ \vert X^1_{t} - \bar{X}_{t}^1 \vert^2 \bigr] = {\mathcal O}(N^{-1}).
\end{equation*}
It is then plain to deduce that 
\begin{equation*}
\limsup_{N \rightarrow + \infty} J^N(\underline{ \bar{\alpha}}) \leq J.
\end{equation*}
 This completes the proof.
\end{proof}

\subsection{Approximate Equilibriums with Distributed Closed Loop Controls}
When $\sigma$ doesn't depend upon $\alpha$, we are able to provide an approximate equilibrium using only distributed controls in closed loop form. This is of real interest from the practical point of view. Indeed, in a such case, the optimizer $\hat{\alpha}$ of the Hamiltonian, as defined in 
\eqref{fo:alphahat}, doesn't depend on $z$. It thus reads as $\hat{\alpha}(t,x,\mu,y)$.
By Proposition \ref{prop:value function}, this says that the optimal control
$(\alpha_{t})_{0 \leq t \leq T}$ in Theorem \ref{th:sol:mkv} has the \emph{feedback} form:
\begin{equation}
\label{eq:3:2:3}
\alpha_{t} = \hat{\alpha}(t,X_{t},\mu_{t},v(t,X_{t})), \quad t \in [0,T].
\end{equation}
The reader might object that, also in the case when $\sigma$ depends upon $\alpha$, the process $Z_{t}$ at time $t$ is also expected to read as a function 
of $X_{t}$, since such a representation is known to hold in the classical \emph{decoupled} forward-backward setting. Even if we feel that it is indeed possible to prove such a representation in our more general setting, we must address the following points: $(i)$ From a practical point of view, \eqref{eq:3:2:3} is meaningful if the feedback function is Lipschitz-continuous, as the Lipschitz property ensures that the stochastic differential equation obtained by plugging \eqref{eq:3:2:3} into the forward equation in \eqref{fo:mkv_fbsde} is solvable; $(ii)$ In the current framework, the function $v$ is known to be Lipschitz continuous by Proposition \ref{prop:value function}, but proving the same result for the representation of $Z_{t}$ in terms of $X_{t}$ seems to be really challenging (by the way, it is already challenging in the standard case, i.e. without any McKean-Vlasov interaction); $(iii)$ We finally mention that, in any case, the relationship between $Z_{t}$ and $X_{t}$, if exists, must be rather intricate as $Z_{t}$ is expected to solve the equation $Z_{t} = \partial_{x} v(t,X_{t}) \sigma(t,X_{t},\PP_{X_{t}},\hat{\alpha}(t,X_{t},\PP_{X_{t}},Y_{t},Z_{t}))$, which can be formally derived by identifying martingale integrands when expanding $Y_{t}=v(t,X_{t})$ by a formal application of It\^o's formula. This equation has been investigated in \cite{WuZhen} in the standard case, but we feel more convenient not to repeat this analysis in the current setting in order to keep things at a reasonable level of complexity. 
\vspace{4pt}

Now, for each integer $N$, we can consider the solution $(X_{t}^1,\dots,X_{t}^N)_{0 \leq t \leq T}$ of the system of $N$ stochastic differential equations
\begin{equation}
\label{fo:gameN}
dX_{t}^i = b\bigl(t,X_{t}^i,\mu^N_{t},\hat{\alpha}\bigl(t,X_{t}^i,\mu_{t},v(t,X_{t}^i)\bigr) \bigr) dt + \sigma(t,X_{t}^i,\mu_{t}^N) dW_{t}^i, 
 \qquad \mu^N_{t} = \frac{1}{N} \sum_{j=1}^N \delta_{X_{t}^j},
\end{equation}
with $t \in [0,T]$ and $X_{0}^i = x_{0}$. The system  (\ref{fo:gameN})  is well posed since  $v$ satisfies 
Proposition \ref{prop:value function} and the minimizer $\hat{\alpha}(t,x,\mu_{t},y)$ is Lipschitz continuous and at most of linear growth in the variables $x$, $\mu$ and $y$, uniformly in $t \in [0,T]$.
The processes $(X^i)_{1 \leq i \leq N}$ give the dynamics of the private states of the $N$ players in the stochastic differential game of interest when the players use the strategies
\begin{equation}
\label{fo:alphaNi}
\alpha_{t}^{N,i} = \hat{\alpha}(t,X_{t}^i,\mu_{t},v(t,X_{t}^i)),\qquad 0\le t\le T,\;\; i\in\{1,\cdots,N\}.
\end{equation}
These strategies are in closed loop form. They are even \emph{distributed} since, at each time $t \in [0,T]$, a player only needs to know the state of his own private state in order to compute the value of the action to take at that time. By the linear growth of $v$ and of the minimizer $\hat\alpha$,  it holds, for any $p \geq 2$,
\begin{equation}
\label{eq:9:6:1}
\sup_{N \geq 1} \max_{1 \leq i \leq N} \bigl[ {\mathbb E} \bigl[ \sup_{0 \leq t \leq T} \vert X_{t}^i \vert^{p} \bigr]
\bigr]< + \infty,
\end{equation}
the expectation inside the brackets being actually independent of $i$ since the strategy is obviously exchangeable. 

We then have the approximate equilibrium property:

\begin{theorem}
\label{th:equilibrium}
In addition to assumptions (B1-4), assume that $\sigma$ doesn't depend upon $\alpha$. Then, 
\begin{equation*}
J^N(\u \beta) \geq J^N(\u \alpha^N) - {\mathcal O}(N^{-1/(d+4)}),
\end{equation*}
for any exchangeable $\u \beta=(\beta^1,\cdots,\beta^N)$, where $\alpha$ is defined in \eqref{fo:alphaNi}. 
\end{theorem}

\begin{proof} We use the same notations as in the proof of Theorem \ref{th:limit cost}.

Since $\bar{\alpha}^1_{t}$ now reads as $\hat{\alpha}(t,\bar{X}_{t}^1,\mu_{t},v(t,\bar{X}_{t}^1))$ for $0 \leq t \leq T$,
we first notice, by the growth property of $v$, that 
$\EE[\sup_{0 \leq t \leq T} \vert \bar{X}_{t}^1 \vert^p] < +\infty$ for any $p \geq 1$. 
As mentioned in 
\eqref{fo:horowitz} in Remark \ref{re:convergence}, this says that  
$\ell_{N}(d)$ in the lower bound 
\begin{equation*}
J^N(\u \beta) \geq J - C  \ell_{N}(d),
 \end{equation*}
see \eqref{eq:3:2:2},
can be chosen as $N^{-1/(d+4)}$.  
  
Moreover, since $v(t,\cdot)$ is Lipschitz continuous, using once again classical estimates from the theory of propagation of chaos (see e.g. \cite{Sznitman} or \cite{JourdainMeleardWoyczynski}), we also have
\begin{equation*}
 \sup_{0 \leq t \leq T} \EE \bigl[ \vert X^i_{t} - \bar{X}_{t}^i \vert^2 \bigr] =  \sup_{0 \leq t \leq T} \EE \bigl[ \vert X^1_{t} - \bar{X}_{t}^1 \vert^2 \bigr] = {\mathcal O}(N^{-1}),
\end{equation*}
so that 
\begin{equation*}
 \sup_{0 \leq t \leq T} \EE \bigl[ \vert \alpha^{N,i}_{t} - \bar{\alpha}_{t}^{i} \vert^2 \bigr]  =  \sup_{0 \leq t \leq T} \EE \bigl[ \vert \alpha^{N,1}_{t} - \bar{\alpha}_{t}^{1} \vert^2 \bigr] = {\mathcal O}(N^{-1}),
\end{equation*}
for any $1 \leq i \leq N$.
It is then plain to deduce that 
\begin{equation*}
J^N(\alpha^N) \leq J + C \ell_{N}(d).
\end{equation*}
This completes the proof.
\end{proof}

\bibliographystyle{plain}
\bibliography{games}

\end{document}